\documentclass[12 pt]{amsart}
\usepackage{geometry,amsmath,amssymb, amsthm, epic}
\usepackage[enableskew]{youngtab}

\newcommand{\onebar}{\overline{1}}
\newcommand{\twobar}{\overline{2}}
\newcommand{\threebar}{\overline{3}}
\newcommand{\fourbar}{\overline{4}}
\newcommand{\fivebar}{\overline{5}}
\newcommand{\kbar}{\overline{k}}

\newcommand{\dsp}{\displaystyle}

\newcommand{\udots}{{\mathinner{\mskip1mu\raise1pt\vbox{\kern7pt\hbox{.}}\mskip2mu\raise4pt\hbox{.}\mskip2mu\raise7pt\hbox{.}\mskip1mu}}}

\newcommand{\eqnref}[1]{(\ref{#1})}
\newtheorem{definition}{Definition}
\newtheorem{lemma}{Lemma}
\newtheorem{proposition}{Proposition}
\newtheorem{corollary}{Corollary}
\newtheorem{bodytheorem}{Theorem}
\newtheorem{theorem}{Theorem}
\newtheorem{conjecture}{Conjecture}
\newtheorem*{nntheorem}{Theorem}

\theoremstyle{remark}
\newtheorem{remark}{Remark}

\DeclareMathOperator{\gen}{gen}
\DeclareMathOperator{\ord}{ord}
\DeclareMathOperator{\str}{str}
\DeclareMathOperator{\row}{row}
\DeclareMathOperator{\height}{hgt}
\DeclareMathOperator{\wgt}{wt}
\DeclareMathOperator{\barred}{bar}
\DeclareMathOperator{\con}{con}
\DeclareMathOperator{\inv}{inv}
\DeclareMathOperator{\pr}{pr}

\begin{document}

\title{Weyl Group Multiple Dirichlet Series of Type C}

\author{Jennifer Beineke}
\address{Department of Mathematics \\ Western New England College \\ Springfield, MA 01119}
\email{jbeineke@wnec.edu}
\author{Ben Brubaker}
\address{Department of Mathematics \\ MIT \\ Cambridge, MA 02139}
\email{brubaker@math.mit.edu}
\author{Sharon Frechette}
\address{Department of Mathematics and Computer Science \\ College of the Holy Cross \\ Worcester, MA 01610}
\email{sfrechet@mathcs.holycross.edu}


\begin{abstract} We develop the theory of ``Weyl group multiple Dirichlet series'' for root systems of type $C$. For an arbitrary root system of rank $r$ and a positive integer~$n$, these are Dirichlet series in $r$ complex variables with analytic continuation and functional equations isomorphic to the associated Weyl group. In type~$C$, they conjecturally arise from the Fourier-Whittaker coefficients of minimal parabolic Eisenstein series on an $n$-fold metaplectic cover of $SO(2r+1)$. For any odd $n$, we construct an infinite family of Dirichlet series conjecturally satisfying the above analytic properties. The coefficients of these series are exponential sums built from Gelfand-Tsetlin bases of certain highest weight representations. Previous attempts to define such series by Brubaker, Bump, and Friedberg in \cite{WMD2} and \cite{WMD4} required $n$ to be sufficiently large, so that coefficients could be described by Weyl group orbits. We prove that our Dirichlet series equals that of \cite{WMD2} and \cite{WMD4} in the case where both series are defined, and hence inherits the desired analytic properties for $n$ sufficiently large. Moreover our construction is valid even for $n=1$, where we prove our series is a Whittaker coefficient of an Eisenstein series. This requires the Casselman-Shalika formula for unramified principal series and a remarkable deformation of the Weyl character formula of Hamel and King \cite{hamelking}.
\end{abstract}

\maketitle


\section{Introduction}\label{sec:intro}

Let $\Phi$ be a reduced root system of rank $r$. ``Weyl group multiple Dirichlet series'' (associated to $\Phi$) are Dirichlet series in $r$ complex variables which initially converge on a cone in $\mathbb{C}^r$, possess analytic continuation to a meromorphic function on the whole complex space, and satisfy functional equations whose action on $\mathbb{C}^r$ is isomorphic to the Weyl group of $\Phi$.

For various choices of $\Phi$ and a positive integer $n$, infinite families of Weyl group multiple Dirichlet series defined over any number field $F$ containing the $2n^{\text{th}}$ roots of unity were introduced in~{\cite{WMD4}}, \cite{WMD3}, \cite{CG}, and \cite{CGgen}. The coefficients of these Dirichlet series are intimately related to the $n^{\text{th}}$ power reciprocity law in $F$. It is further expected that these families are related to metaplectic Eisenstein series as follows.  If one considers the split, semisimple, simply connected algebraic group $G$ over $F$ whose Langlands $L$-group has root system $\Phi$, then it is conjectured that the families of multiple Dirichlet series associated to $\Phi$ 
are precisely the Fourier-Whittaker coefficients of minimal parabolic Eisenstein series on the $n$-fold metaplectic cover of~$G$. 

In light of this suggested relationship with Eisenstein series, one should be able to provide definitions of multiple Dirichlet series for any reduced root system $\Phi$ and any positive integer $n$ having the desired analytic properties. However a satisfactory theory of such Dirichlet series, linked to metaplectic Eisenstein series, has only recently emerged for type $A$. This paper improves the current theory by  developing some of the corresponding results for type~$C$, suggesting that such representations of Eisenstein series should hold in great generality.   After reviewing several definitions below, the remainder of this introduction will be devoted to a brief account of the previously known results about Weyl group multiple Dirichlet series cited above, followed by a discussion of the main results of this paper.

For any reduced root system $\Phi$ of rank $r$, the basic shape of the Weyl group multiple Dirichlet series can be described uniformly in terms of quantities attached to the root system. Given a number field $F$ containing the $2n^{\text{th}}$ roots of unity and a finite set of places $S$ of $F$ (chosen with certain restrictions described in  Section~\ref{subsec:algprelim}), let $\mathcal{O}_S$ denote the ring of $S$-integers in $F$ and $\mathcal{O}_S^\times$ the units in this ring. Then to any $r$-tuple of non-zero $\mathcal{O}_S$ integers $\mathbf{m} = (m_1,\ldots,m_r)$, we associate a ``Weyl group multiple Dirichlet series'' in $r$ complex variables $\mathbf{s}=(s_1, \ldots, s_r)$ of the form 
\begin{equation} Z_\Psi(s_1, \ldots, s_r;m_1, \ldots, m_r) = Z_\Psi(\mathbf{s};\mathbf{m}) = \sum_{\mathbf{c}=(c_1,\ldots,c_r) \in (\mathcal{O}_S / \mathcal{O}_S^\times)^r} \frac{H^{(n)}(\mathbf{c};\mathbf{m}) \Psi(\mathbf{c})}{|c_1|^{2s_1} \cdots |c_r|^{2s_r}} \label{eq:zshape} \end{equation}
where $H^{(n)}(\mathbf{c};\mathbf{m})$ is an arithmetically interesting function to be defined, $\Psi(\mathbf{c})$ is taken from a finite-dimensional complex vector space defined precisely in Section~\ref{sec:Kubota} and guarantees the numerator of our series is well-defined up to $\mathcal{O}_S^\times$ units, and $|c_i| = |c_i |_S$ denotes the norm of the integer~$c_i$ as a product of local norms in $F_S = \prod_{v \in S} F_v$.

The coefficients $H^{(n)}(\mathbf{c};\mathbf{m})$ are not multiplicative, but nearly so and (as we will demonstrate in  (\ref{cmult}) and (\ref{mmult}) of Section~\ref{subsec:twistedmult}) can nevertheless  be reconstructed from coefficients of the form
\begin{equation} H^{(n)}(p^{\mathbf{k}}; p^{\mathbf{l}}) := H^{(n)}(p^{k_1},\ldots,p^{k_r};p^{l_1},\ldots,p^{l_r}) \label{Hppart} \end{equation}
where $p$ is a fixed prime in $\mathcal{O}_S$ and $k_i = \ord_p(c_i)$, $l_i = \ord_p(m_i)$. 

There are two approaches to defining these prime-power contributions. In \cite{CG} and~\cite{CGgen}, Chinta and Gunnells use a remarkable action of the Weyl group to define the coefficients in (\ref{Hppart}) as an average over elements of the Weyl group for any root system~$\Phi$ and any integer $n \geq 1$, from which functional equations and analytic continuation of the series $Z$ follow. By contrast, in \cite{WMD3}, for $\Phi$ of type $A$ and any $n \geq 1$, Brubaker, Bump, and Friedberg define the prime-power coefficients as a sum over basis vectors in a highest weight representation associated to the fixed $r$-tuple~$\mathbf{l}$ in~(\ref{Hppart}). They subsequently prove functional equations and analytic continuation for the multiple Dirichlet series via intricate combinatorial arguments in \cite{gelbvol} and \cite{gtconj}. More recently, the definition in \cite{WMD3} has been shown to match a simpler definition for the prime power coefficients offered in \cite{WMD4} that applies for any root system $\Phi$, {\bf but only for $\mathbf{n}$ sufficiently large} depending on $\Phi$ and $\mathbf{l}$ (see~(\ref{eqn:stability}) for the precise inequality). It is therefore natural to ask whether a definition for the prime power coefficients (\ref{Hppart}) in the mold of \cite{WMD3} (i.e.\ expressible as a sum over basis vectors of highest weight representations) exists for every root system $\Phi$, so that the resulting multiple Dirichlet series possesses good analytic properties (continuation, functional equations) for any $n \geq 1$, and which matches the definition in \cite{WMD4} for $n$ sufficiently large. 

In the case of $\Phi$ of type $C$, the above discussion leads to the following conjecture:
\begin{conjecture} \label{ZConj} For $\Phi = C_r$ for any $r$ and for $n$ odd, the Dirichlet series $Z_\Psi(\mathbf{s};\mathbf{m})$ described in (\ref{eq:zshape}) above, with coefficients of the form $H^{(n)}(p^{\mathbf{k}}; p^{\mathbf{l}})$ as defined in Section~\ref{sec:ppowercoeffs}, has the following properties:
\begin{enumerate}
 \item[I.] $Z_\Psi(\mathbf{s};\mathbf{m})$ possesses analytic continuation to a meromorphic function on $\mathbb{C}^r$ and satisfies a group of functional equations isomorphic to $W(Sp(2r))$, the Weyl group of $Sp(2r)$, of the form (\ref{simplereflectionfe})
 where the $W$ action on $\mathbb{C}^r$ is as given in~(\ref{saction}).
 \item[II.] $Z_\Psi(\mathbf{s};\mathbf{m})$ is the Whittaker coefficient of a minimal parabolic Eisenstein series on an $n$-fold metaplectic cover of $SO_{2r+1}(F_S)$.
\end{enumerate}
\end{conjecture}

Note that part II of the conjecture would imply part I according to the general Langlands-Selberg theory of Eisenstein series extended to metaplectic covers as in~\cite{moeglin-wald}. In practice, other methods to prove part I have resulted in sharp estimates for the scattering matrix involved in the functional equations that would be difficult to obtain from the general theory (see, for example,~{\cite{WMD2}}).

In this paper, we make progress toward this general conjecture by proving the following two results. These will be restated more precisely in the later sections once careful definitions have been given.

\begin{theorem} For $n$ sufficiently large (as given in (\ref{eqn:stability})), $Z_\Psi(\mathbf{s};\mathbf{m})$ matches the multiple Dirichlet series defined in~{\cite{WMD4}} for the root system $\Phi=C_r$. Therefore, for such odd~$n$, the multiple Dirichlet series possess the analytic properties cited in part I of Conjecture~\ref{ZConj}.
\end{theorem}

\begin{theorem}\label{thm:matchHK}
For $n=1$, $Z_\Psi(\mathbf{s};\mathbf{m})$ is a multiplicative function whose prime power coefficients match those of the Casselman-Shalika formula for $Sp(2r)$, hence agreeing with the minimal parabolic (non-metaplectic) Eisenstein series for $SO_{2r+1}(F_S)$. This gives both parts of Conjecture~\ref{ZConj} for $n=1$.
\end{theorem}

These theorems are the symplectic analogue of those proven for type $A$ in~{\cite{WMD3}} and~{\cite{WMD4}}. Theorem 2 is proved using a combinatorial identity of Hamel and King~{\cite{hamelking}}. Theorem~1, our main result,  also has a combinatorial proof using rather subtle connections between the Weyl group and Gelfand-Tsetlin patterns (henceforth $GT$-patterns), which parametrize basis vectors for highest weight representations of $Sp(2r,\mathbb{C})$, the Langlands dual group of $SO(2r+1)$.

\begin{remark} The restriction that $n$ must be odd is natural in light of earlier work by Savin \cite{Savin} showing that the structure of the Iwahori-Hecke algebra depends on the parity of the metaplectic cover and by Bump, Friedberg, and Ginzburg \cite{BFG} on conjectural dual groups for metaplectic covers. Indeed, though the construction of the Dirichlet series we propose in Section~3 makes sense for any $n$, attempts to prove functional equations for $n$ even and $\mathbf{m}$ fixed using the techniques of \cite{edinburgh} suggest the coefficients have the wrong shape. In view of this evidence, we expect a similar combinatorial definition to hold for $n$ even, but making use of the highest weight representation theory for $SO(2r+1, \mathbb{C})$.
\end{remark}

As noted above, the analogue of Conjecture~\ref{ZConj} is known for type $A$ for any $n \geq 1$.  A combinatorial proof of the type $A$ analogue of part I using only rank 1 Eisenstein series is completed in \cite{gelbvol} and \cite{gtconj}. The proof there makes critical use of the outer automorphism of the Dynkin diagram for type $A$, so a simple mimicking of the proof techniques to obtain results for type $C$ is not possible. However, given any fixed $\mathbf{m}$ and fixed $n$, one can in practice confirm the functional equations with  a finite amount of checking (see, for example, \cite{edinburgh} for the details of this argument in a small rank example). 

The type $A$ analogue of part II of Conjecture 1 is proved in \cite{eisencon} by computing the Fourier-Whittaker coefficients of Eisenstein series directly by inducing from successive maximal parabolics. The result is essentially a complicated recursion involving exponential sums and lower rank Eisenstein series. Then one checks the definition given in~{\cite{WMD3}} satisfies the recursion. We expect a similar approach may be possible in type $C$ as well, and this will be the subject of future work. Note that such an approach depends critically on having a proposed solution to satisfy the recursion, so the methods of this paper are a necessary first step.

The precise definition of the prime-power coefficients~(\ref{Hppart}) for type $C$ are somewhat complicated, so we have chosen to postpone the definition until Section 3.  As alluded to earlier, coefficients $H^{(n)}(p^{\mathbf{k}};p^{\mathbf{l}}))$ will be described in terms of basis vectors for highest weight representations of $Sp(2r)$ with highest weight corresponding to~$\mathbf{l}$. As noted in Remark 2 of Section 3, the definition produces Gauss sums which encode subtle information about Kashiwara raising/lowering operators in the crystal graph associated to the highest weight representation. As such, this paper offers the first evidence that mysterious connections between metaplectic Eisenstein series and crystal bases may hold in much greater generality, persisting beyond the type $A$ theory in \cite{gtconj}, \cite{eisencon} and \cite{WMD3}. These connections may not be properly understood until a general solution to our problem for all root systems $\Phi$ is obtained.

Finally, the results of this paper give infinite classes of Dirichlet series with analytic continuation. One can then use standard Tauberian techniques to extract mean-value estimates for families of number-theoretic quantities appearing in the numerator of the series (or the numerator of polar residues of the series). For the $n$-cover of $A_r$, this method yielded the mean-value results of \cite{Chinta} ($r=5, n=2$) and \cite{bba3} ($r=3, n=3$). It would be interesting to explore similar results in type $C$ (remembering that our conjecture may be verified for any given example with $n$, $r$, and $\mathbf{m}$ fixed with only a finite amount of checking, as sketched in \cite{edinburgh}).

We are grateful to Dan Bump, Gautam Chinta, Sol Friedberg, and Paul Gunnells for sharing drafts of manuscripts in progress and for numerous illuminating mathematical conversations. This work was partially supported by NSF grants DMS-0502730 (Beineke), and DMS-0702438 and DMS-0652529 (Brubaker).


\section{Definition of the Multiple Dirichlet Series}\label{sec:defineMDS}

In this section, we present general notation for root systems and the corresponding Weyl group multiple Dirichlet series.

\subsection{Root Systems}

Let $\Phi$ be a reduced root system contained in $V$, a real vector space of dimension $r$. The dual vector space $V^\vee$ contains a root system $\Phi^\vee$ in bijection with $\Phi$, where the bijection switches long and short roots. If we write the dual pairing
\begin{equation} V \times V^\vee \longrightarrow \mathbb{R}: \quad (x,y) \mapsto B(x,y), \label{bilinearform} \end{equation}
then
$ B(\alpha, \alpha^{\vee}) = 2 $. Moreover, the simple reflection $\sigma_\alpha : V \rightarrow V$ corresponding to $\alpha$ is given by
$$ \sigma_\alpha(x) = x - B(x, \alpha^\vee) \alpha. $$
Note that $\sigma_\alpha$ preserves $\Phi$. Similarly, we define $\sigma_{\alpha^\vee}: V^\vee \rightarrow V^\vee$ by $\sigma_{\alpha^\vee}(x) = x - B(\alpha,x) \alpha^\vee$ with $\sigma_{\alpha^\vee}(\Phi^\vee)=\Phi^\vee$.

For our purposes, without loss of generality, we may take $\Phi$ to be irreducible (i.e., there do not exist orthogonal subspaces $\Phi_1, \Phi_2$ with $\Phi_1 \cup \Phi_2 = \Phi$). Then set $\langle \cdot, \cdot \rangle$ to be the Euclidean inner product on $V$ and $||\alpha|| = \sqrt{\langle \alpha, \alpha \rangle}$ the Euclidean norm, where we normalize so that $2\langle \alpha, \beta \rangle$ and $|| \alpha ||^2$ are integral for all $\alpha, \beta \in \Phi$. With this notation,
\begin{equation} \sigma_\alpha(\beta) = \beta - \frac{2 \langle \beta, \alpha \rangle}{\langle \alpha, \alpha \rangle} \alpha \quad \text{for any $\alpha, \beta \in \Phi$} \label{eq:sigmaaction} \end{equation}

We partition $\Phi$ into positive roots $\Phi^+$ and negative roots $\Phi^-$ and let $\Delta = \{ \alpha_1, \ldots, \alpha_r \} \subset \Phi^+$ denote the subset of simple positive roots. Further, we will denote the fundamental dominant weights by $\epsilon_i$ for $i=1, \ldots, r$ satisfying
\begin{equation} \frac{2 \langle \epsilon_i, \alpha_j \rangle}{\langle \alpha_j, \alpha_j \rangle} = \delta_{ij} \quad \text{$\delta_{ij}$ : Kronecker delta.} \label{fundwts} \end{equation}
Any dominant weight $\lambda$ is expressible in terms of the $\epsilon_i$, and a distinguished role in the theory is played by the Weyl vector $\rho$, defined by
\begin{equation} \rho = \frac{1}{2} \sum_{\alpha \in \Phi^+} \alpha = \sum_{i=1}^r \epsilon_i. \label{eqn:weylvect} \end{equation}

\subsection{Algebraic Preliminaries}\label{subsec:algprelim}
In keeping with the foundations used in previous papers (cf. \cite{WMD2} and \cite{WMD4}) on Weyl group multiple Dirichlet series, we choose to define our Dirichlet series as indexed by integers rather than ideals. By using this approach, the coefficients of the Dirichlet series will closely resemble classical exponential sums, but some care needs to be taken to ensure the resulting series remains well-defined up to units. 

To this end, we require the following definitions. Given a fixed positive odd integer $n$, let $F$ be a number field containing the $2n^{\text{th}}$ roots of unity, and let $S$ be a finite set of places containing all ramified places over $\mathbb{Q}$, all archimedean places, and enough additional places so that the ring of $S$-integers $\mathcal{O}_S$ is a principal ideal domain. Recall that the $\mathcal{O}_S$ integers are defined as
$$ \mathcal{O}_S = \left\{ a \in F \; \mid \; a \in \mathcal{O}_v \; \forall v \, \not\in S \right\}, $$
and can be embedded diagonally in
$$ F_S = \prod_{v \in S} F_v. $$
There exists a pairing 
$$ ( \cdot , \cdot )_S \, : \, F_S^\times \times F_S^\times \longrightarrow \mu_n \text{ defined by } (a,b)_S = \prod_{v \in S} (a,b)_v, $$
where the $(a,b)_v$ are local Hilbert symbols associated to $n$ and $v$.

Further, to any $a \in \mathcal{O}_S$ and any ideal $\mathfrak{b} \subseteq \mathcal{O}_S$, we may associate the $n$th power residue symbol $\left( \frac{a}{\mathfrak{b}} \right)_n$ as follows. For prime ideals $\mathfrak{p}$, the expression $\left( \frac{a}{\mathfrak{p}} \right)_n$ is the unique $n^{\text{th}}$ root of unity satisfying the congruence
$$ \left( \frac{a}{\mathfrak{p}} \right)_n \equiv a^{(N(\mathfrak{p})-1)/n} \; (\text{mod } \mathfrak{p}). $$
We then extend the symbol to arbitrary ideals $\mathfrak{b}$ by multiplicativity, with the convention that the symbol is 0 whenever $a$ and $\mathfrak{b}$ are not relatively prime. Since $\mathcal{O}_S$ is a principal ideal domain by assumption, we will write
$$ \left( \frac{a}{b} \right)_n = \left( \frac{a}{\mathfrak{b}} \right)_n \quad \text{for $\mathfrak{b} = b\mathcal{O}_S$} $$
and often drop the subscript $n$ on the symbol when the power is understood from context.

Then if $a,b$ are coprime integers in $\mathcal{O}_S$, we have the $n$th power reciprocity law (cf. \cite{Neukirch}, Thm. 6.8.3)
\begin{equation} \left( \frac{a}{b} \right) = (b,a)_S \left( \frac{b}{a} \right) \label{powerrecip} \end{equation}
which, in particular, implies that if $\epsilon \in \mathcal{O}_S^\times$ and $b \in \mathcal{O}_S$, then
$$ \left( \frac{\epsilon}{b} \right) = (b,\epsilon)_S. $$

Finally, for a positive integer $t$ and $a,c \in \mathcal{O}_S$ with $c \ne 0$, we define the Gauss sum $g_t(a,c)$ as follows. First, choose a non-trivial additive character $\psi$ of $F_S$ trivial on the $\mathcal{O}_S$ integers (cf. \cite{BB:kubota} for details). Then the $n^{\text{th}}$-power Gauss sum is given by
\begin{equation} g_t(a,c) = \sum_{d \; \text{mod } c} \left( \frac{d}{c} \right)_n^t \psi \left( \frac{ad}{c} \right), \label{gausssum} \end{equation} 
where we have suppressed the dependence on $n$ in the notation on the left. The Gauss sum $g_t$ is not multiplicative, but rather satisfies
\begin{equation} g_t(a, c c') = \left( \frac{c}{c'} \right)_n^t \left( \frac{c'}{c} \right)_n^t g_t(a,c) g_t(a,c') \label{gausssummult} \end{equation}
for any relatively prime pair $c,c' \in \mathcal{O}_S$.

\subsection{Kubota's Rank 1 Dirichlet series}\label{sec:Kubota}

Many of the definitions for Weyl group multiple Dirichlet series are natural extensions of those from the rank 1 case, so we begin with a brief description of these.

A subgroup $\Omega \subset F_S^\times$ is said to be {\em isotropic} if $(a,b)_S=1$ for all $a,b \in \Omega$. In particular, $\Omega = \mathcal{O}_S (F_S^\times)^n$ is isotropic (where $(F_S^\times)^n$ denotes the $n^{\text{th}}$ powers in $F_S^\times$). Let $\mathcal{M}_t(\Omega)$ be the space of functions $\Psi : F_S^\times \longrightarrow \mathbb{C}$ that satisfy the transformation property
\begin{equation} \label{rankonepsi}
 \Psi(\epsilon c) = (c,\epsilon)_S^{-t} \Psi(c) \quad \text{for any $\epsilon \in \Omega, c \in F_S^\times$.}
\end{equation}
 
For $\Psi \in \mathcal{M}_t(\Omega)$, consider the following generalization of Kubota's Dirichlet series: 
\begin{equation}\label{eqn:KubotaDS}
 \mathcal{D}_t(s, \Psi, a) = \sum_{0 \neq c \in \mathcal{O}_s/\mathcal{O}^\times_s} \frac{g_t(a,c) \Psi(c)}{ |c|^{2s}}.
\end{equation}
Here $|c|$ is the order of $\mathcal{O}_S / c\mathcal{O}_S$, $g_t(a,c)$ is as in (\ref{gausssum}) and the term $g_t(a,c) \Psi(c) |c|^{-2s}$ is independent of the choice of representative $c$, modulo $S$-units.  Standard estimates for Gauss sums show that the series is convergent if $\mathfrak{R}(s) > \frac{3}{4}$.
Our functional equation computations will hinge on the functional equation for this Kubota Dirichlet series.  Before stating this result, we require some additional notation.  Let 
\begin{equation}\label{eqn:gammafactors}
 \mathbf{G}_n(s) = (2\pi)^{-2(n-1)s} n^{2ns} \prod_{j=1}^{n-2} \Gamma \left( 2s-1+\frac{j}{n} \right). 
 \end{equation}
In view of the multiplication formula for the Gamma function, we may also write
\[
 \mathbf{G}_n(s) = (2\pi)^{-(n-1)(2s-1)} \frac{\Gamma(n(2s-1))}{\Gamma(2s-1)}.
\]
Let 
\begin{equation}\label{eqn:fullkubotaDS}
 \mathcal{D}_t^\ast(s,\Psi,a) = \mathbf{G}_m(s)^{[F:\mathbb{Q}]/2} \zeta_F(2ms-m+1) \mathcal{D}_t(s,\Psi,a),
\end{equation}
where $m = n/\gcd(n,t)$, $\frac{1}{2}[F:\mathbb{Q}]$ is the number of archimedean places of the totally complex field $F$, and $\zeta_F$ is the Dedekind zeta function of $F$.  

If $v \in S_{fin}$ let $q_v$ denote the cardinality of the residue class field $\mathcal{O}_v/\mathcal{P}_v$, where $\mathcal{O}_v$ is the local ring in $F_v$ and $\mathcal{P}_v$ is its prime ideal.  By an \emph{$S$-Dirichlet polynomial} we mean a polynomial in $q_v^{-s}$ as $v$ runs through the finite number of places in $S_{fin}$.  If $\Psi \in \mathcal{M}_t(\Omega)$ and $\eta \in F_S^\times$, denote
\begin{equation}
 \widetilde{\Psi}_{\eta}(c) = (\eta,c)_S \, \Psi(c^{-1} \eta^{-1}).
\end{equation}
Then we have the following result (Theorem 1 in \cite{WMD4}), which follows from the work of Brubaker and Bump \cite{BB:kubota}.
\begin{nntheorem}[Brubaker-Bump] \label{thm:kubotaFE}
 Let $\Psi \in \mathcal{M}_t(\Omega)$ and $a \in \mathcal{O}_S$.  Let $m = n/\gcd(n,t)$.  Then $\mathcal{D}_t^\ast(s,\Psi,a)$ has meromorphic continuation to all $s$, analytic except possibly at $s = \frac{1}{2} \pm \frac{1}{2m}$, where it might have simple poles.  There exist $S$-Dirichlet polynomials $P_\eta^t(s)$ depending only on the image of $\eta$ in $F_S^\times / (F_S^{\times})^n$ such that
\begin{equation}\label{eqn:kubotaFE}
 \mathcal{D}_t^\ast(s,\Psi,a) = |a|^{1-2s} \sum_{\eta \in F_S^\times / (F_S^{\times})^n} P_{a\eta}^t(s) \mathcal{D}_t^\ast(1-s,\widetilde{\Psi}_\eta, a).
\end{equation}
\end{nntheorem}
This result, based on ideas of Kubota \cite{kubota2}, relies on the theory of Eisenstein series.  The case $t=1$ is handled in \cite{BB:kubota}; the general case follows as discussed in the proof of Proposition 5.2 of \cite{WMD2}.  Notably, the factor $|a|^{1-2s}$ is independent of the value of $t$.
\smallbreak

\subsection{The form of higher rank multiple Dirichlet series}\label{subsec:twistedmult}

We now begin explicitly defining the multiple Dirichlet series, retaining our previous notation. By analogy with the rank 1 definition in (\ref{rankonepsi}), given an isotropic subgroup $\Omega$, let $\mathcal{M}(\Omega^r)$ be the space of functions $\Psi : (F_S^\times)^r \longrightarrow \mathbb{C}$ that satisfy the transformation property
\begin{equation} \label{generalpsi} \Psi(\mathbf{\epsilon} \mathbf{c}) = \left( \prod_{i=1}^r (\epsilon_i,c_i)_S^{||\alpha_i||^2} \prod_{i < j} (\epsilon_i, c_j)_S^{2\langle \alpha_i, \alpha_j \rangle} \right) \Psi(\mathbf{c}) \end{equation}
for all $\mathbf{\epsilon} = (\epsilon_1, \ldots, \epsilon_r) \in \Omega^r$ and all $\mathbf{c} = (c_1, \ldots, c_r) \in (F_{S}^{\times})^r$.

Recall from the introduction that, given a reduced root system $\Phi$ of fixed rank $r$, an integer $n \geq 1$, $\mathbf{m} \in \mathcal{O}_S^r$, and $\Psi \in \mathcal{M}(\Omega^r)$, we consider a function of $r$ complex variables $\mathbf{s} = (s_1, \ldots, s_r) \in \mathbb{C}^r$ of the form
$$ Z_\Psi(s_1, \ldots, s_r;m_1, \ldots, m_r) = Z_\Psi(\mathbf{s};\mathbf{m}) = \sum_{\mathbf{c}=(c_1,\ldots,c_r) \in (\mathcal{O}_S / \mathcal{O}_S^\times)^r} \frac{H^{(n)}(\mathbf{c};\mathbf{m}) \Psi(\mathbf{c})}{|c_1|^{2s_1} \cdots |c_r|^{2s_r}}. $$

 The function $H^{(n)}(\mathbf{c}; \mathbf{m})$ carries the main arithmetic content. It is not defined as a multiplicative function, but rather a ``twisted multiplicative'' function. For us, this means
that for $S$-integer vectors $\mathbf{c}, \mathbf{c}' \in (\mathcal{O}_S / \mathcal{O}_S^\times)^r$ with $\gcd (c_1 \cdots c_r, c_1' \cdots c_r')=1$, 
\begin{equation} H^{(n)}(c_1c_1', \ldots, c_rc_r'; \mathbf{m}) = \mu(\mathbf{c}, \mathbf{c}') H^{(n)}(\mathbf{c}; \mathbf{m}) H^{(n)}(\mathbf{c}'; \mathbf{m}) \label{cmult} \end{equation}
where $ \mu(\mathbf{c}, \mathbf{c}')$ is an $n^{\text{th}}$ root of unity depending on $\mathbf{c}, \mathbf{c}'$. It is given precisely by
\begin{equation}  \mu(\mathbf{c}, \mathbf{c}') = \prod_{i=1}^r \left( \frac{c_i}{c_i'} \right)_n^{|| \alpha_i ||^2} \left( \frac{c_i'}{c_i} \right)_n^{|| \alpha_i ||^2} \prod_{i < j}  \left( \frac{c_i}{c_j'} \right)_n^{2 \langle \alpha_i, \alpha_j \rangle} \left( \frac{c_i'}{c_j} \right)_n^{2 \langle \alpha_i, \alpha_j \rangle} \label{mudef} \end{equation}
where $\left( \frac{\cdot}{\cdot} \right)_n$ is the $n^{\text{th}}$ power residue symbol defined in Section~\ref{subsec:algprelim}. Note that in the special case $\Phi=A_1$, the twisted multiplicativity in (\ref{cmult}) and (\ref{mudef}) agrees with the identity for Gauss sums in (\ref{gausssummult}) in accordance with the numerator for the rank one case given in (\ref{eqn:KubotaDS}). 

\begin{remark}
We often think of twisted multiplicativity as the appropriate generalization of multiplicativity for the metaplectic group. In particular, for $n=1$ we reduce to the usual multiplicativity on relatively prime coefficients. Moreover, many of the global properties of the Dirichlet series follow (upon careful analysis of the twisted multiplicativity and associated Hilbert symbols) from local properties, e.g. functional equations as in~{\cite{WMD2}} and~{\cite{WMD4}}. For more on this perspective, see~{\cite{friedberg}}.
\end{remark}

Note that the transformation property of functions in $\mathcal{M}(\Omega^r)$ in (\ref{generalpsi}) above is motivated by the identity
$$ H^{(n)}(\mathbf{\epsilon} \mathbf{c}; \mathbf{m}) \Psi(\mathbf{\epsilon} \mathbf{c}) = H^{(n)}(\mathbf{c}; \mathbf{m}) \Psi(\mathbf{c}) \quad \text{for all} \; \mathbf{\epsilon} \in \mathcal{O}_S^r, \mathbf{c},  \mathbf{m} \in (F_S^\times)^r. $$
The proof can be verified using the $n^{\text{th}}$ power reciprocity law from Section~\ref{subsec:algprelim}.

Now, given any $\mathbf{m}, \mathbf{m'}, \mathbf{c} \in \mathcal{O}_S^r$ with $\gcd (m_1' \cdots m_r', c_1 \cdots c_r) =1$, we let
\begin{equation} H^{(n)}(\mathbf{c}; m_1 m_1', \ldots, m_r m_r') = \prod_{i=1}^r \left( \frac{m_i'}{c_i} \right)_n^{-|| \alpha_i ||^2} H^{(n)}(\mathbf{c}; \mathbf{m}). \label{mmult} \end{equation}
The definitions  in (\ref{cmult}) and (\ref{mmult}) imply that it is enough to specify the coefficients $H^{(n)}(p^{k_1}, \ldots, p^{k_r}; p^{l_1}, \cdots, p^{l_r})$ for any fixed prime $p$ with $l_i = \ord_p(m_i)$ in order to completely determine $H^{(n)}(\mathbf{c}; \mathbf{m})$ for any pair of $S$-integer vectors $\mathbf{m}$ and $\mathbf{c}$. These prime-power coefficients are described in terms of data from highest-weight representations associated to $(l_1, \cdots, l_r)$ and will be given precisely in Section~\ref{sec:ppowercoeffs}.

\subsection{Weyl group actions}\label{sec:weylact}

In order to precisely state a functional equation for the Weyl group multiple Dirichlet series, we require an action of the Weyl group $W$ of $\Phi$ on the complex parameters $(s_1, \ldots, s_r)$. This arises from the linear action of $W$, realized as the group generated by the simple reflections $\sigma_{\alpha^\vee}$, on $V^\vee$. From the perspective of Dirichlet series, it is more natural to consider this action shifted by $\rho^\vee$, half the sum of the positive co-roots. Then each $w \in W$ induces a transformation $V_\mathbb{C}^\vee = V^\vee \otimes \mathbb{C} \rightarrow V_\mathbb{C}^\vee$ (still denoted by~$w$) if we require that
$$ B(w \alpha, w(\mathbf{s}) - \frac{1}{2} \rho^\vee) = B(\alpha, \mathbf{s} - \frac{1}{2} \rho^\vee). $$

We introduce coordinates on $V_\mathbb{C}^\vee$ using simple roots $\Delta = \{ \alpha_1, \ldots, \alpha_r \}$ as follows. Define an isomorphism $V_\mathbb{C}^\vee \rightarrow \mathbb{C}^r$ by
\begin{equation} \mathbf{s} \mapsto (s_1, s_2, \ldots, s_r) \quad s_i = B(\alpha_i, \mathbf{s}). \label{impaction} \end{equation}
This action allows us to identify $V_\mathbb{C}^\vee$ with $\mathbb{C}^r$, and so the complex variables $s_i$ that appear in the definition of the multiple Dirichlet series may be regarded as coordinates in either space. It is convenient to describe this action more explicitly in terms of the~$s_i$ and it suffices to consider simple reflections which generate $W$. Using  the action of the simple reflection $\sigma_{\alpha_i}$ on the root system $\Phi$ given in~\eqref{eq:sigmaaction}
 in conjunction with (\ref{impaction}) above gives the following:

\begin{proposition} The action of $\sigma_{\alpha_i}$ on $\mathbf{s} = (s_1, \ldots, s_r)$ defined implicitly in (\ref{impaction}) is given by 
\begin{equation} s_j \mapsto s_j - \frac{2 \langle \alpha_j, \alpha_i \rangle}{\langle \alpha_i, \alpha_i \rangle} \left( s_i - \frac{1}{2} \right) \quad j=1,\ldots,r. \label{saction} \end{equation}
In particular, $\sigma_{\alpha_i} : s_i \mapsto 1-s_i$.
\end{proposition}

\subsection{Normalizing factors and functional equations}

The multiple Dirichlet series must also be normalized using Gamma and zeta factors in order to state precise functional equations. Let
$$ n(\alpha) = \frac{n}{\gcd(n,||\alpha||^2)}, \quad \alpha \in \Phi^+. $$
For example, if $\Phi = C_r$ and we normalize short roots to have length 1, this implies that $n(\alpha) = n$ unless $\alpha$ is a long root and $n$ even (in which case $n(\alpha)=n/2$). By analogy with the zeta factor appearing in (\ref{eqn:fullkubotaDS}), for any $\alpha \in \Phi^+$, let
$$ \zeta_\alpha(\mathbf{s}) = \zeta \left( 1+2n(\alpha) B(\alpha, \mathbf{s} - \frac{1}{2} \rho^\vee) \right) $$
where $\zeta$ is the Dedekind zeta function attached to the number field $F$. Further, for~$\mathbf{G}_n(s)$ as in~\eqref{eqn:gammafactors}, we may define
\begin{equation} \mathbf{G}_\alpha(\mathbf{s}) = \mathbf{G}_{n(\alpha)} \left( \frac{1}{2} + B(\alpha, \mathbf{s} - \frac{1}{2} \rho^\vee) \right). \label{eqn:normfactors}
\end{equation}
Then for any $\mathbf{m} \in \mathcal{O}_S^r$, the normalized multiple Dirichlet series is given by
\begin{equation} Z_\Psi^\ast(\mathbf{s}; \mathbf{m}) = \left[ \prod_{\alpha \in \Phi^+} \mathbf{G}_\alpha(\mathbf{s}) \zeta_\alpha(\mathbf{s}) \right] Z_\Psi(\mathbf{s}, \mathbf{m}). \label{zstardefined} \end{equation}

By considering the product over all positive roots, we guarantee that the other zeta and Gamma factors are permuted for each simple reflection $\sigma_i \in W$, and hence for all elements of the Weyl group.  

Given any fixed $n$, $\mathbf{m}$ and root system $\Phi$, we seek to exhibit a definition for $H^{(n)}(\mathbf{c};\mathbf{m})$ (or equivalently, given twisted multiplicativity, a definition of $H$ at prime-power coefficients) such that $Z_\Psi^\ast(\mathbf{s}; \mathbf{m})$ satisfies functional equations of the form:
\begin{equation} \label{simplereflectionfe}
 Z_\Psi^\ast(\mathbf{s}; \mathbf{m}) = |m_i|^{1-2s_i} Z_{\sigma_i \Psi}^\ast(\sigma_i \mathbf{s}; \mathbf{m})
\end{equation}
for all simple reflections $\sigma_i \in W$. Here, $\sigma_i \mathbf{s}$ is as in (\ref{saction}) and the function $\sigma_i \Psi$, which essentially keeps track of the rather complicated scattering matrix in this functional equation, is defined as in (37) of \cite{WMD4}. As noted in Section 7 of \cite{WMD4}, given functional equations of this type, one can obtain analytic continuation to a meromorphic function of $\mathbb{C}^r$ with an explicit description of polar hyperplanes.


\section{Definition of the Prime-Power Coefficients}\label{sec:ppowercoeffs}

In this section, we give a precise definition of the coefficients $H^{(n)}(p^{\mathbf{k}};p^{\mathbf{l}})$ needed to complete the description of the multiple Dirichlet series for root systems of type~$C_r$ and $n$ odd. All the previous definitions are stated in sufficient generality for application to multiple Dirichlet series for any reduced root system $\Phi$ and any positive integer $n$. Only the prime power coefficients require specialization to our particular root system $\Phi = C_r$, though this remains somewhat complicated. We summarize the definition at the end of the section.

The vector $\mathbf{l} = (l_1, l_2, \ldots, l_r)$ appearing in $H^{(n)}(p^{\mathbf{k}};p^{\mathbf{l}})$ can be associated to a dominant integral element for $Sp_{2r}(\mathbb{C})$ of the form
\begin{equation}\label{eqn:lambda}
\lambda = (l_1 + l_2 + \cdots + l_r, \ldots, l_1+l_2, l_1). 
\end{equation}
The contributions to $H^{(n)}(p^{\mathbf{k}};p^{\mathbf{l}})$ will then be parametrized by basis vectors of the highest weight representation of highest weight $\lambda + \rho$, where $\rho$ is the Weyl vector for $C_r$ defined in~\eqref{eqn:weylvect}, so that
\begin{equation} \lambda + \rho = (l_1 + l_2 + \cdots + l_r+r, \ldots, l_1+l_2+2, l_1+1) =: (L_r, \cdots, L_1). \label{GTtoprow} \end{equation}
In~{\cite{WMD3}}, prime-power coefficients for multiple Dirichlet series of type $A$ were attached to Gelfand-Tsetlin patterns, which parametrize highest weight vectors for~$SL_{r+1}(\mathbb{C})$ (cf.~{\cite{GT}}). Here, we use an analogous basis for the symplectic group, according to branching rules given by Zhelobenko in~{\cite{zhelobenko}}. We will continue to  refer to the objects comprising this basis as Gelfand-Tsetlin patterns, or $GT$-patterns.

More precisely, a $GT$-pattern $P$ has the form
\begin{equation}\label{eqn:GTpattern}
P = \begin{array}{cccccccc}
 a_{0,1} & & a_{0,2} & & \cdots & & a_{0,r} & \\
 & b_{1,1} & & b_{1,2} & \cdots & b_{1,r-1} & & b_{1,r} \\
 & & a_{1,2} & & \cdots & & a_{1,r} & \\
 & & & \ddots & & \ddots & & \vdots \\
 & & & & & & a_{r-1,r} & \\
 & & & & & & & b_{r,r}
\end{array}
\end{equation}
where the $a_{i,j}, b_{i,j}$ are non-negative integers and the rows of the pattern interleave. That is, for all $a_{i,j}, b_{i,j}$ in the pattern $P$ above,
$$ \min(a_{i-1,j}, a_{i,j}) \geq b_{i,j} \geq \max(a_{i-1,j+1},a_{i,j+1}) $$
and 
$$ \min(b_{i+1,j-1}, b_{i,j-1}) \geq a_{i,j} \geq \max(b_{i+1,j},b_{i,j}). $$
The set of all patterns with top row $(a_{0,1}, \ldots, a_{0,r})=(L_r, \ldots, L_1)$ form a basis for the highest weight representation with highest weight $\lambda + \rho$. Hence, we will consider $GT$-patterns with top row $(L_r, \ldots, L_1)$ as in (\ref{GTtoprow}), and refer to this set of patterns as $GT(\lambda+\rho)$.

The contributions to each $H^{(n)}(p^{\mathbf{k}};p^{\mathbf{l}})$ with both $\mathbf{k}$ and $\mathbf{l}$ fixed come from a single weight space corresponding to $\mathbf{k} = (k_1, \ldots, k_r)$ in the highest weight representation~$\lambda+\rho$ corresponding to $\mathbf{l}$. We first describe how to associate a weight vector to each $GT$-pattern. Let
\begin{equation}\label{def:si}
s_a(i) := \sum_{m=i+1}^r a_{i,m} \qquad \mbox{and} \qquad s_b(i) := \sum_{m=i}^r b_{i,m}
\end{equation}
be the row sums for the respective rows of $a$'s and $b$'s in $P$. (Here we understand that $s_a(r) = 0$ corresponds to an empty sum.)  Then define the weight vector $\wgt(P) = (\wgt_1(P),\ldots,\wgt_r(P))$ by
\begin{equation}
\wgt_i = \wgt_i(P) = s_a(r-i) - 2 s_b(r+1-i) + s_a(r+1-i), \quad i=1,\ldots,r. \label{eqn:defwgt}
\end{equation}
Note that as the weights are generated in turn, we begin at the bottom of the pattern $P$ and work our way up to the top.
Our prime power coefficients will then be supported at $(p^{k_1}, \ldots, p^{k_r})$ with
\begin{equation}
k_i = \sum_{j=i}^r \left( \wgt_j + L_j \right), \quad i=1,\ldots,r-1, \quad k_r = \wgt_r + L_r, \label{eqn:kweights}
\end{equation}
so that in particular, the $k_i$ are non-negative integers.

In terms of the $GT$-pattern $P$, the reader may check that we have 
$$ k(P) = (k_1(P), k_2(P), \ldots, k_r(P)) \quad \text{with} $$ 
\begin{align}
\begin{split}\label{def:ki}
 k_1(P) & =  s_a(0) - \sum_{m=1}^r \left( s_b(m) - s_a(m) \right) \\
 k_i(P) & =  s_a(0) - 2  \sum_{m=1}^{r+1-i} \left( s_b(m) - s_a(m) \right)  - s_a(r+1-i) + \sum_{m=1}^{r+1-i} a_{0,m}
\end{split}
\end{align}
for $1<i \leq r$. Then we define
\begin{equation}\label{eqn:Hprimedef}
H^{(n)}(p^{\mathbf{k}};p^{\mathbf{l}}) = H^{(n)}(p^{k_1}, \ldots, p^{k_r};p^{l_1}, \ldots, p^{l_r}) = \sum_{\substack{ P \in GT(\lambda+\rho) \\ k(P)=(k_1,\ldots,k_r) }} G(P)  
\end{equation}
where the sum is over all $GT$-patterns $P$ with top row $(L_r, \ldots, L_1)$ as in (\ref{GTtoprow}) satisfying the condition $\mathbf{k}(P)=(k_1,\ldots,k_r)$ and $G(P)$ is a weighting function whose definition depends on the following elementary quantities.

To each pattern $P$, define the corresponding data:
\begin{equation}\label{def:uv} v_{i,j} = \sum_{m=i}^j \left( a_{i-1,m} - b_{i,m} \right), \qquad w_{i,j} = \sum_{m=j}^r \left( a_{i,m} - b_{i,m} \right), \qquad  u_{i,j} = v_{i,r} + w_{i,j}, \end{equation}
where we understand the entries $a_{i,j}$ or $b_{i,j}$ to be 0 if they do not appear in the pattern~$P$. 

\begin{remark} The integers $u_{i,j}$ and $v_{i,j}$ have representation theoretic meaning in terms of Kashiwara raising and lowering operators in the crystal graph associated to the highest weight representation of highest weight $\lambda+\rho$ for $U_q(\mathfrak{sp}(2r))$, the quantized universal enveloping algebra of the Lie algebra $\mathfrak{sp}(2r)$. See Littelmann~\cite{littelmann} for details, particularly Corollary 2 of Section 6. See also \cite{gtconj} and \cite{eisencon} for a more complete description in crystal language, focusing mainly on type $A$. We find this interpretation quite striking in light of the connection to Whittaker models on the metaplectic group. Ultimately, this can be seen as another instance of connections between quantum groups and principal series representations in the spirit of \cite{lusztig}. This is not a perspective we emphasize here, but this potentially exciting connection to past work is worth further exploration.
\end{remark}

To each entry $b_{i,j}$ in $P$, we associate
\begin{equation} \gamma_b(i,j) = \begin{cases} g_{\delta_{jr}+1}(p^{v_{i,j} - 1}, p^{v_{i,j}}) & \mbox{if } b_{i,j} = a_{i-1,j+1}, \\
                                   \phi(p^{v_{i,j}}) & \mbox{if } a_{i-1,j} < b_{i,j} < a_{i-1,j+1}, \, \, \, n | v_{i,j} \cdot (\delta_{jr}+1), \\
				   0 & \mbox{if } a_{i-1,j} < b_{i,j} < a_{i-1,j+1}, \, \, \, n \nmid v_{i,j} \cdot (\delta_{jr}+1), \\
				   q^{v_{i,j}} & \mbox{if } b_{i,j} = a_{i-1,j},
                   \end{cases} \label{gammablong} \end{equation}
where $g_t(p^\alpha, p^\beta)$ is an $n^\text{th}$-power Gauss sum as in \eqnref{gausssum}, $\phi(p^a)$ denotes the Euler phi function for $\mathcal{O}_S / p^a \mathcal{O}_S$, $q = | \mathcal{O}_S / p \mathcal{O}_S |$ and $\delta_{jr} $ is the Kronecker delta function.
We note these cases may be somewhat reduced, using elementary properties of Gauss sums, to
\begin{equation} \gamma_b(i,j) = \begin{cases}  q^{v_{i,j}} & \mbox{if } b_{i,j} = a_{i-1,j}, \\
	g_{\delta_{jr}+1}(p^{v_{i,j}+b_{i,j}-a_{i-1,j+1}- 1}, p^{v_{i,j}}) & \text{else.}
	\end{cases} \label{gammabshort} \end{equation}
To each entry $a_{i,j}$ in $P$, with $i \geq 1$, we may associate
\begin{equation} \gamma_a(i,j) = \begin{cases} g_1(p^{u_{i,j} - 1}, p^{u_{i,j}}) & \mbox{if } a_{i,j} = b_{i,j-1}, \\
	\phi(p^{u_{i,j}}) & \mbox{if }  b_{i,j} < a_{i,j} < b_{i,j-1}, \, \, \, n | u_{i,j}, \\
	0 & \mbox{if } b_{i,j} < a_{i,j} < b_{i,j-1}, \, \, \, n \nmid u_{i,j}, \\
	q^{u_{i,j}} & \mbox{if } a_{i,j} = b_{i,j},
	\end{cases} \label{gammaalong} \end{equation}
which can similarly be compacted to
\begin{equation} \gamma_a(i,j) = \begin{cases} q^{u_{i,j}} & \mbox{if } a_{i,j} = b_{i,j}, \\
	g_1(p^{u_{i,j}-a_{i,j}+b_{i,j-1}- 1}, p^{u_{i,j}}) & \text{else.}
	\end{cases} \label{gammaashort} \end{equation}
We introduce terminology to describe relationships between elements in a pattern $P$:
\begin{definition} \label{def:minmaxgen} A $GT$-pattern $P$ is {\em minimal at $b_{i,j}$}
 if $b_{i,j} = a_{i-1,j}$.  It is {\em maximal at $b_{i,j}$} if $1 \leq j < r$ and $b_{i,j} = a_{i-1,j+1}$, or if $b_{i,r} = 0$. If none of these equalities holds, we say $P$ is {\em generic at} $b_{i,j}$.  
 
 Likewise, $P$ is {\em minimal at $a_{i,j}$}
 if $a_{i,j} = b_{i,j}$, and {\em maximal at $a_{i,j}$} if $a_{i,j} = b_{i,j-1}$.  If neither equality holds, we say $P$ is {\em generic at} $a_{i,j}$.
\end{definition}

\begin{definition}\label{def:strict}  A $GT$-pattern $P$ is {\em strict} if its entries are strictly decreasing across each horizontal row.
\end{definition}

We then define the coefficients
\begin{equation}\label{eqn:defineGP}
 G(P) = \begin{cases}
 \prod_{1 \leq i \leq j \leq r} \gamma_a(i,j) \gamma_b(i,j) & \text{if $P$ is strict,}\\
 0 & \text{otherwise},
 \end{cases}
\end{equation}
where we again understand  $\gamma_a(r,r)$ to be 1 since $a_{r,r}$ is not in the pattern $P$. Combining these definitions, we obtain a definition of the prime-power coefficients in the series as summarized below.

\begin{definition}[Summary of definitions for $H$] \label{def:ppart} Given any prime $p$, define
\begin{equation} H^{(n)}(p^\mathbf{k}; p^\mathbf{l}) = \sum_{\substack{ P \in GT(\lambda+\rho) \\ k(P) = \mathbf{k} }} G(P) \label{summarydef} \end{equation}
where the sum is over all $GT$-patterns with top row corresponding to $\lambda+\rho$ and row sums fixed according to (\ref{def:ki}), and $G(P)$ is given as in (\ref{eqn:defineGP}) above with $\gamma_a(i,j)$ and $\gamma_b(i,j)$ of (\ref{gammaashort}) and (\ref{gammabshort}), resp., defined in terms of $v_{i,j}$ and $u_{i,j}$ in (\ref{def:uv}).
\end{definition}

Note that in the right-hand side of \eqnref{summarydef}, we have suppressed the dependence on~$n$.  This is appropriate since the expressions in (\ref{gammabshort}) and (\ref{gammaashort}) are given in terms of Gauss sums, which are defined uniformly for all $n$.

The coefficients $H^{(n)}(\mathbf{c};\mathbf{m})$ appearing in (\ref{eq:zshape}) are now implicitly defined by (\ref{summarydef}) together with the twisted multiplicativity given in (\ref{cmult}) and (\ref{mmult}). The resulting multiple Dirichlet series $Z_\Psi(\mathbf{s};\mathbf{m})$ is initially absolutely convergent for $\Re(s_i)$ sufficiently large. Indeed, if a pattern $P$ has weight $\mathbf{k} = (k_1, \ldots, k_r)$, then $|G(P)| < q^{k_1+ \cdots + k_r}$ and the number of patterns in a given weight space is bounded as a function of $\mathbf{m}$ corresponding to the highest weight vector.


\section{Comparison in the Stable Case}\label{sec:stablecase}

We now compare our multiple Dirichlet series, having $p^{\text{th}}$-power coefficients as defined in (\ref{summarydef}), with the multiple Dirichlet series defined for arbitrary root systems $\Phi$ in~\cite{WMD4}, when $n$ is sufficiently large. In this section, we determine the necessary lower bound on $n$ explicitly, according to a {\em stability assumption} introduced in \cite{WMD2}.  With this lower bound, we can then prove that for $n$ odd, the two prescriptions agree.  

To this end, let $\mathbf{m} = (m_1, \ldots, m_r)$ be a fixed $r$-tuple of non-zero $\mathcal{O}_S$ integers. To any fixed prime $p$ in $\mathcal{O}_S$, set $l_i = \ord_p(m_i)$ for $i=1,\ldots,r$. Then define $\lambda_p$ as in \eqnref{eqn:lambda}, so that in terms of the fundamental dominant weights $\epsilon_i$, we have $$\lambda_p =  \sum_{i=1}^r l_i \epsilon_i.$$ 
Then we may define the function $d_{\lambda_p}$ on the set of positive roots $\Phi^+$ by 
\begin{equation}\label{def:dlambda}
 d_{\lambda_p}(\alpha) = \frac{2 \left<\lambda_p + \rho, \alpha \right>}{\left< \alpha, \alpha \right>}.
\end{equation}
For ease of computation in the results that follow, we choose to normalize the inner product $\langle \ , \ \rangle$ so that $\|\alpha\|^2 = \langle \alpha , \alpha \rangle = 1$ if $\alpha$ is a short root, while $\|\alpha\|^2 = 2$ if $\alpha$ is a long root. 

\noindent {\bf Stability Assumption.}\,\, {\em   Let $\alpha = \sum_{i=1}^r t_i \alpha_i$ be the largest positive root in the partial ordering for $\Phi$.  Then for every prime $p$, we require that the positive integer $n$  satisfies}
\begin{equation}\label{eqn:stability}
n \geq \gcd(n, \|{\alpha}\|^2) \cdot d_{\lambda_p}(\alpha) = \gcd(n, \|{\alpha}\|^2) \cdot \sum_{i=1}^r t_i (l_i + 1).
\end{equation}
When the Stability Assumption holds, we say we are ``in the stable case." Note this is well-defined since $l_i=0$ for all $i=1,\ldots,r$ for all but finitely many primes $p$. For the remainder of this section, we work with a fixed prime $p$, and so write $\lambda$ in place of $\lambda_p$ when no confusion can arise.

For $\Phi = C_r$, we let $\alpha_1$ denote the long simple root, so that the largest positive root is $\alpha_1 + \sum_{i=2}^r 2 \alpha_i$.  Moreover if $n$ is odd, the condition (\ref{eqn:stability}) becomes 
\begin{equation}\label{eqn:nstable}
n \geq l_1 + 1 + \sum_{i=2}^r 2(l_i + 1).
\end{equation}

For any $w \in W(\Phi)$, define the set $\Phi_w = \{ \alpha \in \Phi^+ \mid w(\alpha) \in \Phi^-\}$.  Following~{\cite{WMD2}} and \cite{WMD4}, the $p^{th}$-power coefficients of the multiple Dirichlet series  in the stable case are given by
\begin{equation}\label{eqn:ppartofH}
H_{st}^{(n)}(p^{k_1}, \cdots, p^{k_r}; p^{l_1}, \cdots, p^{l_r}) = \prod_{\alpha \in \Phi_w} g_{\|\alpha\|^2}(p^{d_{\lambda}(\alpha)-1}, \, p^{d_{\lambda}(\alpha)}),
\end{equation}
where the dependence on $n$ occurs only in the $n$th power residue symbol in the Gauss sums. In \cite{WMD4}, it was established that the above definition $H_{st}^{(n)}(p^{\mathbf{k}};p^{\mathbf{l}})$ produces a Weyl group multiple Dirichlet series $Z^\ast(\mathbf{s},\mathbf{m})$ with analytic continuation and functional equations (of the form in Conjecture 1) provided the Stability Assumption on $n$ holds. The proof works for any reduced root system $\Phi$. In this section, we demonstrate that our definition $H^{(n)}(p^\mathbf{k};p^{\mathbf{l}})$ in terms of $GT$-patterns as in (\ref{summarydef}) matches that in (\ref{eqn:ppartofH}) for $n$ satisfying the Stability Assumption as in (\ref{eqn:stability}).

\begin{definition}
 Given a $GT$-pattern $P \in GT(\lambda+\rho)$, and $G(P)$ defined as in (\ref{eqn:defineGP}), then $P$ is said to be {\bf stable} if $G(P) \ne 0$ for some (odd) $n$ satisfying the Stability Assumption as in (\ref{eqn:stability}). 
\end{definition}

As we will see in the following result, if $P$ is stable for one such $n$, then $G(P)$ is non-zero for all $n$ satisfying (\ref{eqn:stability}). These are the relevant patterns we must consider in establishing the equivalence of the two definitions $H_{st}^{(n)}(p^{\mathbf{k}};p^{\mathbf{l}})$ and $H^{(n)}(p^{\mathbf{k}};p^{\mathbf{l}})$ in the stable case, and we begin by characterizing all such patterns.

\begin{proposition}
 A pattern $P \in GT(\lambda+\rho)$ is stable if and only if, in each pair of rows in $P$ with index $i$ (that is, pattern entries $\{ b_{i,j},a_{i,j} \}_{j=i}^r$), the ordered set
 $$ \left\{ b_{i,i}, b_{i,i+1}, \ldots, b_{i,r}, a_{i,r}, a_{i,r-1}, \ldots, a_{i,i+1} \right\} $$
 has an initial string in which all elements are minimal (as in Definition~\ref{def:minmaxgen}) and all remaining elements are maximal.     
\end{proposition}

\begin{proof}
 If any element $a_{i,j}$ or $b_{i,j}$ in the pattern $P$ is neither maximal nor minimal, i.e. is ``generic'' in the sense of Definition~\ref{def:minmaxgen}, then $\gamma_a(i,j)$ (or $\gamma_b(i,j)$, resp.) is non-zero if and only if $n | u_{i,j}$ according to (\ref{gammaalong}) (or $n | v_{i,j}(\delta_{jr}+1)$ according to (\ref{gammablong}), resp.). But one readily checks that $n$ is precisely chosen in the Stability Condition so that $n > \max_{i,j} \{ u_{i,j}, (\delta_{jr}+1) v_{i,j} \}$ and hence neither divisibility condition can be satisfied. Therefore all entries of any stable $P$ must be maximal or minimal. The additional necessary condition that $P$ be strict (as in Definition~\ref{def:strict}) so that $G(P)$ is not always zero according to (\ref{eqn:defineGP}) guarantees that neighboring entries in the ordered set can never be of the form (maximal,minimal), which gives the result.  
\end{proof}

Note that the number of stable patterns $P$ is thus $2^r r! = |W(C_r)|,$ the order of the Weyl group of $C_r$.

\subsection{Action of $W$ on Euclidean space}

In demonstrating the equality of the two prime-power descriptions, we found it necessary to use an explicit coordinatization of the the root system embedded in $\mathbb{R}^r$; it would be desirable to find a coordinate-free proof.
Let $\mathbf{e}_i$ be the standard basis vector (1 in $i$th component, 0 elsewhere) in $\mathbb{R}^r.$ We choose the following coordinates for the simple roots of $C_r$:  
\begin{equation}\label{Crposroots}
 \alpha_1 = 2\mathbf{e}_1, \quad \alpha_2 = \mathbf{e}_2 - \mathbf{e}_1, \quad \ldots, \quad \alpha_r = \mathbf{e}_r - \mathbf{e}_{r-1}.
\end{equation}
Consider an element $w \in W(C_r)$, the Weyl group of $C_r$.  As an action on $\mathbb{R}^r$, this group is generated by all permutations $\sigma$ of the basis vectors $\mathbf{e}_1, \ldots, \mathbf{e}_r$ and all reflections $\mathbf{e}_i \mapsto - \mathbf{e}_i$ for $i=1,\ldots,r$.  Thus we may describe the action explicitly using $\varepsilon_w^{(i)} \in \{+1, -1\}$ for $i=1, 2, \ldots, r$ so that  
\begin{equation}
 w(t_1, t_2, \ldots, t_r) = (\varepsilon_w^{(1)}t_{\sigma^{-1}(1)}\, , \,\varepsilon_w^{(2)}t_{\sigma^{-1}(2)}\, , \, \ldots, \, \varepsilon_w^{(r)}t_{\sigma^{-1}(r)}). \label{waction}
\end{equation}

In the following proposition, we associate a unique Weyl group element $w$ with each $GT$-pattern $P$ that is stable.  In this result, and in the remainder of this section, it will be convenient to refer to the rows of $P$ beginning at the {\em bottom} rather than the top.  We will therefore discuss rows $a_{r-i}$, for $1 \leq i \leq r$, for instance. 

\begin{proposition}\label{prop:wforP}
Let $P$ be a stable strict $GT$-pattern with top row $L_r \ L_{r-1} \ \cdots \ L_1$, hence with associated dominant weight vector $\lambda = \sum_{i=1}^r \ell_i \varepsilon_i$.  Let non-negative integers $k_1(P), \ldots, k_r(P)$ be defined as in \eqnref{def:ki}, and let $k_{r+1}(P) = 0$.  Then there exists a unique element $w \in W(C_r)$ such that
\begin{equation}
 \lambda + \rho - w(\lambda + \rho) = (2k_1-k_2, k_2-k_3, \ldots, k_{r-1}-k_r, k_r) = \sum_{i=1}^r k_i \alpha_i
\end{equation}
In fact, for $i=2, \ldots, r$, 
\begin{equation}
 k_{i+1} - k_i + L_i = -\wgt_i = \varepsilon_w^{(i)} L_{\sigma^{-1}(i)},
\end{equation}
where $L_{\sigma^{-1}(i)}$ is the unique element in row $a_{r-i}$ that is not in row $a_{r+1-i}$, and the weight coordinate $\wgt_i$ is as in~\eqnref{eqn:defwgt}.    Similarly,
\begin{equation}
 k_2 - 2k_1 + L_1 = -\wgt_1 = \varepsilon_w^{(1)} L_{\sigma^{-1}(1)},
\end{equation}
where $L_{\sigma^{-1}(1)}$ is the unique element in row $a_{r-1}$ that is not in row $a_r$.
\end{proposition}

\begin{proof}
The definitions for $\rho$ and $\lambda$ give $\lambda + \rho = (L_1, \ldots, L_r)$ in Euclidean coordinates.  We compute the coordinates of $(\lambda + \rho)-\sum_{i=1}^r k_i \alpha_i$, using \eqnref{def:ki}.  This gives
\begin{equation}
  L_1 + k_2 - 2k_1  = -[s_a(r-1) - 2s_b(r) + s_a(r)] = -\wgt_1
\end{equation}
and similarly, for $i=2, \ldots, r$,
\begin{equation}
  L_i + k_{i+1} - k_i = -[s_a(r-i) - 2s_b(r+1-i) + s_a(r+1-i)] = -\wgt_i,
\end{equation}
so that 
\begin{equation}
 \lambda + \rho - \sum_{i=1}^r k_i \alpha_i  = -(\wgt_1, \wgt_2, \ldots, \wgt_r).
\end{equation}
Each pattern $P$ has a unique weight vector.  Since $P$ is a stable pattern, it is easy to see that the $i$th weight consists of the unique entry that is in row $a_{r-i}$ but not in row $a_{r+1-i}$, with a negative sign if this entry is present in row $b_{r+1-i}$, or a positive sign if not.   Thus the weight vector is simply a permutation of the entries in the top row, with a choice of sign in each entry.  We may find a unique $w$ (whose action is described above), for which 
\begin{equation}
 w(\lambda + \rho) = (\varepsilon_w^{(1)} L_{\sigma^{-1}(1)}, \ldots, \varepsilon_w^{(r)} L_{\sigma^{-1}(r)}) = -\,(\wgt_1, \, \wgt_2, \, \ldots, \, \wgt_r).
\end{equation}
Thus $L_{\sigma^{-1}(i)}$ is the unique element in row $a_{r-i}$ that is not present in row $a_{r+1-i}$.
\end{proof}
\bigbreak

\begin{corollary}\label{lemma:rowa_i}
Let $P$ be a stable strict $GT$-pattern with top row $L_r \ L_{r-1} \ \cdots \ L_1$.  For $1 \leq i \leq r$, the set of elements in row $a_{r-i}$ satisfies the following:
\begin{equation}
\{a_{r-i, r+1-i}\, , \, a_{r-i, r+2-i}\, , \, \ldots \, , \, a_{r-i,r} \} = \{ L_{\sigma^{-1}(i)}\, , \, L_{\sigma^{-1}(i-1)}\, , \, \ldots \, , \, L_{\sigma^{-1}(1)} \}
\end{equation}
\end{corollary}

\begin{proof}
From Proposition \ref{prop:wforP}, $L_{\sigma^{-1}(j)}$ is the unique element in row $a_{r-j}$ that is not in row $a_{r+1-j}$.  Working downwards, we eliminate these elements for $j=i, i+1, \ldots, r$, in order to reach row $a_{r-j}$.  Therefore, we are left with the remaining set.
\end{proof}

\subsection{Agreement of the multiple Dirichlet series}\label{subsec:maintheorem}

\begin{bodytheorem}\label{thm:matchsums}
Let $\Phi = C_r$ and choose a positive integer $n$ such that the Stability Assumption (\ref{eqn:stability}) holds.   
\begin{enumerate}
 \item [(i)] Let $P$ be a stable strict $GT$-pattern, and let $G(P)$ be the product of Gauss sums defined in \eqnref{eqn:defineGP} in Section \ref{sec:defineMDS}.  Let $w$ be the Weyl group element associated to $P$ as in Proposition \ref{prop:wforP}.  Then $$G(P) = \prod_{\alpha \in \Phi_w} g_{\| \alpha \|^2}(p^{d_\lambda(\alpha)-1}, p^{d_\lambda(\alpha)}),$$
matching the definition given in \eqnref{eqn:ppartofH}, with $d_\lambda(\alpha)$ as defined in \eqnref{def:dlambda}.
 \item [(ii)] $H_{st}(c_1, \cdots, c_r;m_1, \cdots m_r) = H^{(n)}(c_1, \cdots, c_r;m_1, \cdots m_r)$.  
\end{enumerate}
That is, the Weyl group multiple Dirichlet series in the twisted stable case is identical to the series defined by the Gelfand-Tsetlin description for $n$ sufficiently large.
\end{bodytheorem}

\begin{remark} Our main conjecture presented in the introduction states that $n$ should be odd. In fact, the proof of the above theorem works for any $n$ satisfying the Stability Assumption, regardless of parity. However, we believe this is an artifact of the relative combinatorial simplicity of the ``stable" coefficients. As noted in Remark 1, one expects a distinctly different combinatorial recipe than the one presented in this paper to hold uniformly for all even $n$.
\end{remark}

\begin{proof}
 It is clear that part (i) implies part (ii), since both coefficients are obtained from their prime-power parts by means of twisted multiplicativity.

In proving part (i), let $P$ be the $GT$-pattern with top row $L_r \, L_{r-1} \cdots L_1$ associated to $w$ by Proposition \ref{prop:wforP}.  We first note that since $P$ is stable, we have $u_{i,j} = 0$ if $P$ is minimal at $a_{i,j}$, and $v_{i,j} = 0$ if $P$ is minimal at $b_{i,j}$.  Thus
\[
 G(P) = \prod_{a_{i,j} \text{ maximal}} g_1(p^{u_{i,j}-1}, p^{u_{i,j}}) \prod_{b_{i,j} \text{ maximal}} g_{\delta_{jr}+1}(p^{v_{i,j}-1},p^{v_{i,j}}),
\]
It suffices to show that the set of Gauss sum exponents $u_{i,j}$ and $v_{i,j}$ at maximal entries in $P$ coincides with the set of $d_\lambda(\alpha)$ as $\alpha$ runs over $\Phi_w$.  (In fact, we show a slightly sharper statement, which matches Gauss sum exponents at maximal entries in pairs of rows of $P$ with values of $d_\lambda(\alpha)$ as $\alpha$ runs over certain subsets of $\Phi_w$.)  

The number of maximal elements in a pair of rows $b_{r+1-i}$ and $a_{r+1-i}$ is described in the next result.   First, we say that $(i,j)$ is an {\em $i$-inversion} for $w^{-1}$ if $j < i$ and $\sigma^{-1}(j) > \sigma^{-1}(i)$.  The number of these pairs, as well as the number of those for which the inequality is preserved rather than inverted, will play an important role in counting  Gauss sums.  To this end, we define the following quantities
\begin{equation}\begin{split}
 \inv_i(w^{-1}) &= \# \{ (i,j) \mid \sigma^{-1}(j) > \sigma^{-1}(i) \  \text{and} \  j < i\},\\
 \pr_i(w^{-1}) &= \# \{ (i,j) \mid \sigma^{-1}(j) < \sigma^{-1}(i) \  \text{and} \  j < i\}. \label{def:invpr}
\end{split}\end{equation}

\begin{proposition}\label{prop:maximal}
Let $P$ be a stable strict $GT$-pattern with top row $L_r \ L_{r-1} \ \cdots \ L_1$, and let $w \in W$ be the Weyl group element associated to $P$ as in Proposition \ref{prop:wforP}.  Let $\inv_i(w)$ and $ \pr_i(w)$ be as defined in~\eqref{def:invpr}, and let $m_i(P)$ denote the number of maximal entries in rows $b_{r+1-i}$ and $a_{r+1-i}$ together.  Then,
\begin{equation}
m_i(P) = \begin{cases}
\inv_i(w^{-1}) & \text{ if } \varepsilon_w^{(i)} = +1,\\
i +  \pr_i(w^{-1})  \ &\text{ if } \varepsilon_w^{(i)} = -1.\end{cases}
\end{equation}
\end{proposition}

\begin{proof}
Recall from our means of associating $w$ to $P$ that $\varepsilon_w^{(i)}$ is opposite in sign from the $i$th Gelfand-Tsetlin weight.  Consider row $b_{r+1-i}$ together with the rows immediately above and below:
{\small \[
\begin{array}{ccccccccc}
 a_{r-i,r+1-i} & & a_{r-1, r+2-i} & & \cdots & &\cdots & a_{r-i,r} & \\
 & b_{r+1-i,r+1-i} & & \cdots & & \cdots & & & b_{r+1-i,r}\\
 & & a_{r+1-i, r+2-i} & & \cdots & & \cdots & a_{r+1-i,r}
\end{array}
\]}

Suppose $\varepsilon_w^{(i)} = +1$, so $L_{\sigma^{-1}(i)}$ is missing from row $a_{r+1-i}$ but present in row $b_{r+1-i}$.  Then there are no maximal entries in row $b_{r+1-i}$, and $m_i$ maximal entries in row $a_{r+1-i}$, so
\begin{equation}\label{eqn:rowb_pos}
 b_{r+1-i, r+j-i} = a_{r-i, r+j-i} \quad \text{ for } 1 \leq j \leq i, 
\end{equation}
and
\begin{equation}\label{eqn:rowa_pos}
a_{r+1-i, r+(j+1)-i} = 
	\begin{cases} b_{r+1-i, r+j-i} \quad &\text{ for } 1 \leq j \leq m_i,\\
	b_{r+1-i,r+(j+1)-i} \quad &\text{ for } m_i+1 \leq j \leq i.
	\end{cases}
\end{equation}
Moreover, the entry $L_{\sigma^{-1}(i)}$ in row $b_{r+1-i}$  marks the switch from maximal to minimal as we move from left to right in row $a_{r+1-i}$.  That is, all entries in row $a_{r+1-i}$ to the left of $L_{\sigma^{-1}(i)}$ are maximal, while all those  to the right are minimal.  By Lemma \ref{lemma:rowa_i}, row $a_{r+1-i}$ consists of the elements in the set $\{ L_{\sigma^{-1}(j)} \mid j < i \}$.  Since the rows of $P$ are strictly decreasing, this means the maximal entries in row $a_{r+1-i}$ are given by\[
\{ L_{\sigma^{-1}(j)} \mid j < i \text{ and } \sigma^{-1}(j) > \sigma^{-1}(i)\} 
\]
This set clearly has order $\inv_i(w^{-1})$.  
\medbreak

Now suppose $\varepsilon_w^{(i)} = -1$, so that $L_{\sigma^{-1}(i)}$ is missing from both row $a_{r+1-i}$ and row $b_{r+1-i}$.  Then all entries in row $a_{r+1-i}$ are maximal, and the last $m_i-i+1$ entries in row $b_{r+1-i}$ are maximal, so
\begin{equation}\label{eqn:rowa_neg}
 a_{r+1-i, r+(j+1)-i} = b_{r+1-i, r+j-i}, \quad \text{ for } 1 \leq j \leq i-1, 
\end{equation}
and
\begin{equation}\label{eqn:rowb_neg}
b_{r+1-i, r+j-i} = 
	\begin{cases} a_{r-i, r+j-i} \quad &\text{ for } 1 \leq j \leq 2i-1-m_i, \\
	a_{r-i,r+(j+1)-i} \quad &\text{ for } 2i-m_i \leq j \leq i-1,\\
	0 \quad &\text{ for } j = i.
	\end{cases}
\end{equation}
The entry $L_{\sigma^{-1}(i)}$ in row $a_{r-i}$  marks the switch from minimal to maximal as we move to the right in row $b_{r+1-i}$.  That is, all entries below and to the left of $L_{\sigma^{-1}(i)}$ are minimal, while those below and to the right are maximal.  Since rows $b_{r+1-i}$ and $a_{r+1-i}$ are identical, the entries of row $b_{r+1-i}$ are $\{ L_{\sigma^{-1}(j)} \mid j < i \}$, by Lemma \ref{lemma:rowa_i}.  Moreover, since rows are strictly decreasing, the maximal entries in row $b_{r+1-i}$ are given by
\[
\{ L_{\sigma^{-1}(j)} \mid j < i \text{ and } \sigma^{-1}(j) < \sigma^{-1}(i)\} \cup \{ 0\}
\]
This set has order $\pr_i(w^{-1}) + 1$.  Counting maximal entries in both rows, we obtain $m_i = (i-1) +  \pr_i(w^{-1}) + 1 = i + \pr_i(w^{-1})$.
\end{proof}

Next, we establish a finer characterization of $\Phi_w = \{ \alpha \in \Phi^+ \, \mid \, w(\alpha) \in \Phi^- \}$.  For $\Phi = C_r$, the roots in $\Phi^+$ take different forms; the positive long roots are $2 \mathbf{e}_\ell$ for $1 \leq \ell \leq r$, while the positive short roots are $\mathbf{e}_m \pm \mathbf{e}_\ell$ for $1 \leq \ell < m \leq r$.  We will express $\Phi_w$ as a disjoint union of subsets indexed by $i \in \{1, 2, \ldots, r\}$.  To this end, let $i$ be fixed, and let $j$ be any positive integer such that $j < i$. Consider positive roots of the following three types:
\begin{eqnarray*}
 \text{ Type L : }  &\alpha_{i,w} :=& 2 \mathbf{e}_{\sigma^{-1}(i)}.\\
 \text{Type S$^+$ : } &\alpha_{i,j,w}^+ :=& \mathbf{e}_{\sigma^{-1}(j)} + \mathbf{e}_{\sigma^{-1}(i)}.\\
 \text{Type S$^-$ : } &\alpha_{i,j,w}^- :=& 
	\begin{cases}  \mathbf{e}_{\sigma^{-1}(j)} - \mathbf{e}_{\sigma^{-1}(i)}\, \text{ if }\, \sigma^{-1}(j) > \sigma^{-1}(i), \\ \mathbf{e}_{\sigma^{-1}(i)} - \mathbf{e}_{\sigma^{-1}(j)}\, \text{ if }\, \sigma^{-1}(j) < \sigma^{-1}(i). \end{cases}
\end{eqnarray*}
Clearly we encounter each positive root exactly once as $i$ and $j$ vary as indicated.  Let $\Phi_w^{(i)} \subseteq \Phi_w$ denote the set of all $\alpha_{i,w}$, $\alpha_{i,j,w}^+$, $\alpha_{i,j,w}^-$  belonging to $\Phi_w$.  The following lemma completely characterizes $\Phi_w^{(i)}$.    

\begin{lemma}\label{lem:Phi_w^i}
Let $i \in \{1, 2, \ldots, r\}$ be fixed, let $j$ be any positive integer with $j < i$, and let $\Phi_w^{(i)}$ be as defined above.  Then 
\begin{enumerate}
 \item [(1)] $\alpha_{i,w} \in \Phi_w^{(i)}$ if and only if $\varepsilon_w^{(i)} = -1$.
 \item [(2)] $\alpha_{i,j,w}^- \in \Phi_w^{(i)}$ if and only if $\sigma^{-1}(j) < \sigma^{-1}(i)$ and $\varepsilon_w^{(i)} = -1$, or $\sigma^{-1}(j) > \sigma^{-1}(i)$ and $\varepsilon_w^{(i)} = +1$.
 \item [(3)] $\alpha_{i,j,w}^+ \in \Phi_w^{(i)}$ if and only if $\varepsilon_w^{(i)} = -1$.
\end{enumerate}
Consequently,  $|\Phi_w^{(i)}| = m_i(P)$, as defined in Proposition \ref{prop:maximal}.  
\end{lemma}

\begin{proof}
As defined in (\ref{waction}), $w$ acts on a basis vector $\mathbf{e}_{\ell}$ simply as $w(\mathbf{e}_{\ell}) = \varepsilon_w^{(\ell)}\, \mathbf{e}_{\sigma(\ell)}$, and this action extends linearly to each of the roots.  Part (1) is immediate from the definition of $\Phi_w$.  
\smallbreak

For part (2), if $\sigma^{-1}(j) < \sigma^{-1}(i)$ then $w(\alpha_{i,j,w}^-) = \varepsilon_w^{(i)}\mathbf{e}_i - \varepsilon_w^{(j)} \mathbf{e}_j$.  If $\varepsilon_w^{(i)} = +1$, then since $j < i$, we  have $w(\alpha_{i,j,w}^-) \in \Phi^+$ regardless of the value of $\varepsilon_w^{(j)}$.  Thus $\alpha_{i,j,w}^- \notin \Phi_w^{(i)}$.  Similarly, if $\varepsilon_w^{(i)} = -1$, then since $j < i$, we  have $w(\alpha_{i,j,w}^-) \in \Phi^-$ regardless of the value of $\varepsilon_w^{(j)}$.  Thus $\alpha_{i,j,w}^- \in \Phi_w^{(i)}$.

On the other hand, if $\sigma^{-1}(j) > \sigma^{-1}(i)$ then $w(\alpha_{i,j,w}^-) = \varepsilon_w^{(j)}\mathbf{e}_j - \varepsilon_w^{(i)} \mathbf{e}_i$.  Considering the cases $\varepsilon_w^{(i)} = +1, -1$ in turn, we find that regardless of the value of $\varepsilon_w^{(j)}$, we have $w(\alpha_{i,j,w}^-) \in \Phi_w^{(i)}$ if and only if $\varepsilon_w^{(i)} = +1$.
\smallbreak

For part (3), we have $w(\alpha_{i,j,w}^+) = \varepsilon_w^{(j)}\mathbf{e}_j + \varepsilon_w^{(i)} \mathbf{e}_i$.  Using a similar argument, we see that independently of the value of $\varepsilon_w^{(j)}$, $w(\alpha_{i,j,w}^+)$ is a negative root when $\varepsilon_w^{(i)}$ is negative, and a positive root otherwise.  
\smallbreak

Finally, we count elements in $\Phi_w^{(i)}$.  If $\varepsilon_w^{(i)} = +1$, the conditions yield $\inv_i(w^{-1})$ elements of type S$^-$, and zero elements of types L and S$^+$.  On the other hand, if $\varepsilon_w^{(i)} = -1$, there is one element of type L, $i-1$ elements of type S$^+$, and $\pr_i(w^{-1})$ elements of type S$^-$.  In either case, $|\Phi_w^{(i)}| = m_i(P)$.
\end{proof}

\bigbreak
For each of the roots in $\Phi^{(i)}_w$, we compute the corresponding $d_\lambda$ (as defined in (\ref{def:dlambda})) below. 
\begin{lemma}\label{lem:dlambda}
 With the notation as above, we have 
\begin{enumerate}
 \item [(1)] $d_\lambda(\alpha_{i,w}) = L_{\sigma^{-1}(i)}$.
 \item [(2)] $d_\lambda(\alpha_{i,j,w}^-) = \begin{cases}  L_{\sigma^{-1}(j)} - L_{\sigma^{-1}(i)}\, \text{ if }\, \sigma^{-1}(j) > \sigma^{-1}(i), \\ L_{\sigma^{-1}(i)} - L_{\sigma^{-1}(j)}\, \text{ if }\, \sigma^{-1}(j) < \sigma^{-1}(i). \end{cases}$
 \item [(3)] $d_\lambda(\alpha_{i,j,w}^+) = L_{\sigma^{-1}(j)} + L_{\sigma^{-1}(i)}$.
\end{enumerate}

\end{lemma}

\begin{proof}
First, we compute $d_\lambda(\alpha_{i,w}) = d_\lambda(2 \mathbf{e}_{\sigma^{-1}(i)})$.  Using \eqnref{Crposroots}, we have
	\begin{equation}
	 \alpha_{i,w} = \alpha_1 + \sum_{k=2}^{\sigma^{-1}(i)} 2 \alpha_k,
	\end{equation}
where we regard the sum to be $0$ if $\sigma^{-1}(i) = 1$.  Since $\left< \alpha_{i,w}, \alpha_{i,w} \right> = \left< \alpha_1, \alpha_1 \right> = 2$ and $\left< \alpha_k, \alpha_k \right> = 1$ for $k = 2, \ldots, r$, we have 
	\begin{equation}\label{eqn:dlam_long}
	 d_\lambda(\alpha_{i,w}) = \frac{2 \left< \lambda + \rho, \alpha_{i,w} \right>}	
	 {\left< \alpha_{i,w}, \alpha_{i,w} \right>} = \sum_{m=1}^r ( l_m+1) 
	 \sum_{k=1}^{\sigma^{-1}(i)} \frac{2 \left< \epsilon_m, \alpha_k \right>}{\left< \alpha_k, \alpha_k \right>} 
	 = L_{\sigma^{-1}(i)}
	\end{equation}

Next, we compute $d_\lambda(\alpha_{i,j,w}^-) = d_\lambda(\mathbf{e}_{\sigma^{-1}(i)} - \mathbf{e}_{\sigma^{-1}(j)})$ if $\sigma^{-1}(j) < \sigma^{-1}(i)$.  (The computations if $\sigma^{-1}(j) > \sigma^{-1}(i)$ are analogous.)  In this case, \eqnref{Crposroots} gives
	\begin{equation}
	 \alpha_{i,j,w}^- = \sum_{k=\sigma^{-1}(j) + 1}^{\sigma^{-1}(i)} \alpha_k,
	\end{equation}
where the sum is nonempty as $\sigma^{-1}(j) < \sigma^{-1}(i)$.  Since $\left< \alpha_{i,j,w}^-, \alpha_{i,j,w}^- \right> = 1$, we have 
  \begin{equation}\label{eqn:dlam_neg}
	 d_\lambda(\alpha_{i,j,w}^-) = \sum_{m=1}^r ( l_m+1) 
	 \sum_{k=\sigma^{-1}(j)+1}^{\sigma^{-1}(i)} 
	 \frac{2 \left< \epsilon_m, \alpha_k \right>}{\left< \alpha_k, \alpha_k \right>} 
	 = L_{\sigma^{-1}(i)} - L_{\sigma^{-1}(j)}
  \end{equation}

Finally, we compute $d_\lambda(\alpha_{i,j,w}^+) = d_\lambda(\mathbf{e}_{\sigma^{-1}(i)} + \mathbf{e}_{\sigma^{-1}(j)})$.  Here, \eqnref{Crposroots} gives
\begin{equation}
 \alpha_{i,j,w}^+ = \alpha_1 + \sum_{k=2}^{\sigma^{-1}(j)} 2 \alpha_k + \sum_{k=\sigma^{-1}(j) + 1}^{\sigma^{-1}(i)} \alpha_k,
\end{equation}
where the first sum is $0$ if $\sigma^{-1}(j) = 1$.  Since $\left< \alpha_{i,j,w}^+, \alpha_{i,j,w}^+ \right> = 1$ as well, we have
  \begin{equation}\label{eqn:dlam_pos}
 	d_\lambda(\alpha_{i,j,w}^+) = \sum_{m=1}^r ( l_m+1) 
	\left[ \sum_{k=1}^{\sigma^{-1}(j)} \frac{4 \left< \epsilon_m, \alpha_k \right>}
	{\left< \alpha_k, \alpha_k \right>} + \sum_{k=\sigma^{-1}(j)+1}^{\sigma^{-1}(i)} 
	\frac{2 \left< \epsilon_m, \alpha_k \right>}{\left< \alpha_k, \alpha_k \right>} 
	\right] = L_{\sigma^{-1}(i)} + L_{\sigma^{-1}(j)}
  \end{equation}
\end{proof}

Now let $D_i = \{ d_\lambda(\alpha) \mid \alpha \in \Phi_w^{(i)} \}$.  By Lemmas \ref{lem:Phi_w^i} and \ref{lem:dlambda}, we see that \\if $\epsilon_w^{(i)} = +1$, then 
\begin{equation}\label{eqn:Di_plus}
 D_i = \{ L_{\sigma^{-1}(j)} - L_{\sigma^{-1}(i)} \mid j<i \text{ and } \sigma^{-1}(j) > \sigma^{-1}(i) \},
\end{equation}
while if $\epsilon_w^{(i)} = -1$, then 
\begin{equation}\begin{split}\label{eqn:Di_minus}
D_i = \{ L_{\sigma^{-1}(i)} \}\  \cup \ &\{ L_{\sigma^{-1}(j)} + L_{\sigma^{-1}(i)} \mid j<i \} \\
&\cup \ \{ L_{\sigma^{-1}(i)} - L_{\sigma^{-1}(j)} \mid j<i \text{ and } \sigma^{-1}(j) < \sigma^{-1}(i) \}.
                \end{split}
\end{equation}

Now we examine the Gauss sums obtained from the $GT$-pattern $P$ with top row $L_r \, L_{r-1} \, \cdots \, L_1$ associated to $w$.  Suppose there are $m_i = m_i(P)$ maximal entries in rows $b_{r+1-i}$ and $a_{r+1-i}$ combined.  First, suppose there are no maximal entries in row $b_{r+1-i}$.  Then the first $m_i$ entries in row $a_{r+1-i}$ (reading from the left) are maximal.  Since there are $i-1$ entries in row $a_{r+1-i}$, in this case we have $m_i < i$.  We may apply equations \eqnref{eqn:rowb_pos} and \eqnref{eqn:rowa_pos} to compute the sums defining $u_{k,\ell}$ and $v_{k,\ell}$.  These sums telescope, and we have
\begin{eqnarray*}
                 v_{r+1-i\, , \, r+j-i} &=& 0, \quad \text{ for } 1 \leq j \leq i-1,\\
		u_{r+1-i\, , \, r+(j+1)-i} &=& \begin{cases} 0, &  \text{ for } m_i+1 \leq j \leq i,\\
		a_{r-i\, , \, r+j-i} - b_{r+1-i\, , \, r+(m_i+1)-i},  & \text{ for } 1 \leq j \leq m_i.\end{cases}
\end{eqnarray*}
By Proposition \ref{prop:wforP}, $b_{r+1-i\, , \, r+(m_i+1)-i} = L_{\sigma^{-1}(i)}$, so to compute $u_{r+1-i\, , \, r+(j+1)-i}$ as $j$ varies, we must determine the set of values for $a_{r-i\, , \, r+j-i}$ with $1 \leq j \leq m_i$. Recall that by Lemma \ref{lemma:rowa_i}, the entries in row $a_{r-i}$ are given by
\begin{equation}\label{eqn:row_r-i}
 \{ L_{\sigma^{-1}(j)} \mid 1 \leq j \leq i \}
\end{equation}
Since the rows are strictly decreasing, the entries appearing to the left of \\$a_{r-1\, , \, r + (m_i+1)-i} = L_{\sigma^{-1}(i)}$ have an index greater than $\sigma^{-1}(i)$.  That is,  
\begin{equation}
 \{ a_{r-i\, , \, r+j-i} \mid 1 \leq j \leq m_i \} = \{ L_{\sigma^{-1}(j)} \mid j<i \text{ and } \sigma^{-1}(j) > \sigma^{-1}(i) \}
\end{equation}
Thus the nonzero Gauss sum exponents for rows $b_{r+1-i}$ and $a_{r+1-i}$ are given by 
$u_{r+1-i\, , \, r+(j+1)-i} = L_{\sigma^{-1}(j)} - L_{\sigma^{-1}(i)}$ with $j<i \text{ and } \sigma^{-1}(j) > \sigma^{-1}(i)$.  Finally, note that $\varepsilon_w^{(i)} = +1$, since there are no maximal entries in row $b_{r+1-i}$ in this case.  Thus our set of nonzero Gauss sum exponents matches the set $D_i$ as given in \eqnref{eqn:Di_plus}.
\medbreak
Second, suppose there are maximal entries in row $b_{r+1-i}$.  Consequently, all entries in row $a_{r+1-i}$ are maximal, so there are $n_i:= m_i - i + 1$ maximal entries in row $b_{r+1-i}$.  We may apply equations \eqnref{eqn:rowa_neg} and \eqnref{eqn:rowb_neg} to compute the sums defining $u_{k,\ell}$ and $v_{k,\ell}$.  These sums telescope, and we have
\begin{eqnarray*}
                 v_{r+1-i\, , \, r+j-i} &=& \begin{cases} 0, & \text{ for } 1 \leq j \leq i-n_i,\\
                                             a_{r-i\, , \, r+1-n_i} - a_{r-i\, , \, r+(j+1)-i}, & \text{ for } i+1-n_i \leq j \leq i-1\\  a_{r-i\, , \, r+1-n_i} & \text{ for } j=i
                                            \end{cases}\\
		u_{r+1-i\, , \, r+(j+1)-i} &=& a_{r-i\, , \, r+1-n_i} + a_{r+1-i\, , \, r+(j+1)-i},  
				\; \; \; \text{ for } 1 \leq j \leq i-1.
\end{eqnarray*}
By Proposition \ref{prop:wforP}, $a_{r+1-i\, , \, r+1-n_i} = L_{\sigma^{-1}(i)}$, and thus $v_{r+1-i\, , r} = L_{\sigma^{-1}(i)}$.  To compute the remaining exponents $v_{r+1-i\, , \, r+j-i}$ as $j$ varies, we again appeal to \eqnref{eqn:row_r-i}.  Since the rows are strictly decreasing, the entries appearing to the right of $L_{\sigma^{-1}(i)}$ in row $a_{r-1}$ must have an index smaller than $\sigma^{-1}(i)$.  That is, 
\begin{equation}
 \{ a_{r-i\, , \, r+(j+1)-i} \mid i+1-n_i \leq j \leq i-1 \} = \{ L_{\sigma^{-1}(j)} \mid j<i \text{ and } \sigma^{-1}(i) > \sigma^{-1}(j) \}
\end{equation}
Thus $v_{r+1-i\, , \, r+j-i} = L_{\sigma^{-1}(i)} - L_{\sigma^{-1}(j)}$ with $i+1-n_i \leq j < i$ and $\sigma^{-1}(i) > \sigma^{-1}(j)$.  
\medbreak

To compute the exponents $u_{r+1-i\, , \, r+(j+1)-i}$, we note that by Lemma \ref{lemma:rowa_i}, the entries in row $a_{r+1-i}$ are the $L_{\sigma^{-1}(j)}$ for which $1 \leq j \leq i-1$.  Thus $u_{r+1-i\, , \, r+(j+1)-i} = L_{\sigma^{-1}(i)} + L_{\sigma^{-1}(j)}$ with  $\mid 1 \leq j \leq i-1.$  Finally, we note that $\varepsilon_w^{(i)} = -1,$ since there are maximal entries in row $b_{r+1-i}$.  Combining the cases above, we see that we match the set $D_i$ given in \eqnref{eqn:Di_minus}.  
\medbreak

This completes the proof of Theorem \ref{thm:matchsums}.
\end{proof}


\section{Comparison with the Casselman-Shalika formula}

The main focus of this section is the proof of Theorem \ref{thm:matchHK}, using a generating function identity given by Hamel and King~{\cite{hamelking}}.  This identity may be regarded as a deformation of the Weyl character formula for Sp$(2r)$, though it is stated in the language of symplectic, shifted tableaux (whose definition we will recall in this section) so we postpone the precise formulation.
Recall that our multiple Dirichlet series take the form
$$ Z_\Psi(\mathbf{s};\mathbf{m}) = \sum_{\mathbf{c}=(c_1,\ldots,c_r) \in (\mathcal{O}_S / \mathcal{O}_S^\times)^r} \frac{H^{(n)}(\mathbf{c};\mathbf{m}) \Psi(\mathbf{c})}{|c_1|^{2s_1} \cdots |c_r|^{2s_r}}. $$
In brief, we show that for $n=1$ our formulas for the prime power supported contributions of $Z_\Psi(\mathbf{s}, \mathbf{m})$ match one side of Hamel and King's identity, while the other side of the identity is given in terms of a character of a highest weight representation for Sp$(2r)$. By combining the Casselman-Shalika formula with Hamel and King's result, we will establish Theorem \ref{thm:matchHK}.

\subsection{Specialization of the multiple Dirichlet series for $n=1$} \label{specializingsection}

Many aspects of the definition $Z_\Psi$ are greatly simplified when $n=1$. First, we may take $\Psi$ to be constant, since the Hilbert symbols appearing in the definition (\ref{generalpsi}) are trivial for $n=1$. Moreover, the coefficients $H^{(n)}(\mathbf{c};\mathbf{m})$ for $n=1$ are perfectly multiplicative in both $\mathbf{c}$ and $\mathbf{m}$. That is, according to (\ref{mudef}) we have
$$ H^{(1)}(\mathbf{c} \cdot \mathbf{c}';\mathbf{m}) =  H^{(1)}(\mathbf{c};\mathbf{m}) H^{(1)}(\mathbf{c}';\mathbf{m}) \; \text{when $\gcd(c_1 \cdots c_r, c_1' \cdots c_r')=1$} $$ 
and according to (\ref{mmult}) we have
$$ H^{(1)}(\mathbf{c} ; \mathbf{m} \cdot \mathbf{m'}) =  H^{(1)}(\mathbf{c};\mathbf{m}) \; \text{when $\gcd(m_1' \cdots m_r', c_1 \cdots c_r)=1$.} $$
Hence the global definition of $Z_\Psi(\mathbf{s};\mathbf{m})$ for fixed $\mathbf{m}$ is easily recovered from its prime power supported contributions as follows:
\begin{equation} Z_\Psi(\mathbf{s};\mathbf{m}) = \prod_{p \in \mathcal{O}_S} \left[ \sum_{\mathbf{k} = (k_1,\ldots,k_r)} \frac{H^{(1)}(p^\mathbf{k};p^\mathbf{l})}{|p|^{2k_1s_1} \cdots |p|^{2k_rs_r}} \right], \label{pcontrib} \end{equation}
with $\mathbf{l} = (l_1, \cdots, l_r)$ given by $\ord_p(m_i) = l_i$ for $i=1,\ldots,r$. Note that the sum on the right-hand side runs over the finite number of vectors $\mathbf{k}$ for which $H^{(n)}(p^{\mathbf{k}};p^\mathbf{l})$ has non-zero support for fixed $\mathbf{l}$ according to (\ref{summarydef}).

We now simplify our formulas for $H^{(n)}(p^{\mathbf{k}};p^{\mathbf{l}})$ when $n=1$. As before, we set $q = | \mathcal{O}_S / p \mathcal{O}_S |$. With definitions as given in (\ref{gammablong}) and (\ref{gammaalong}), let  
$$ \tilde{\gamma}_a(i,j) := q^{-u_{i,j}} \gamma_a(i,j), \,\, and \,\, \tilde{\gamma}_b(i,j) := q^{-v_{i,j}} \gamma_b(i,j). $$
Then by analogy with the definitions (\ref{eqn:defineGP}) and (\ref{summarydef}), define
$$ \widetilde{G}(P) := \prod_{1 \leq i \leq j \leq r} \tilde{\gamma}_a(i,j) \tilde{\gamma}_b(i,j), $$
and
$$ \widetilde{H}^{(1)}(p^{\mathbf{k}};p^{\mathbf{l}}) = \widetilde{H}^{(1)}(p^{k_1}, \ldots, p^{k_r};p^{l_1}, \ldots, p^{l_r} ) := \sum_{k(P)=(k_1,\ldots,k_r)} \widetilde{G}(P), $$
where again the sum is taken over $GT$-patterns $P$ with fixed top row $(L_r, \cdots, L_1)$ as in (\ref{GTtoprow}).  By elementary properties of Gauss sums, when $n=1$ we have, for a strict $GT$-pattern $P$, 
\begin{equation} \tilde{\gamma}_a(i,j) = \begin{cases} 1 & \text{if $P$ is minimal at $a_{i,j}$}, \\
	1 - \frac{1}{q} & \text{if $P$ is generic at $a_{i,j}$},\\
	- \frac{1}{q} & \text{if $P$ is maximal at $a_{i,j}$}, \end{cases} \label{gammaareduction} \end{equation}
recalling the language of Definition~\ref{def:minmaxgen} and similarly,
\begin{equation} \tilde{\gamma}_b(i,j) = \begin{cases} 1 & \text{if $P$ is minimal at $b_{i,j}$}, \\
	 1 - \frac{1}{q} & \text{if $P$ is generic at $a_{i,j}$}, \\
	- \frac{1}{q} & \text{if $P$ is maximal at $b_{i,j}$}. \end{cases} \label{gammabreduction} \end{equation}
Note that when $P$ is generic at $a_{i,j}$ (resp. $b_{i,j}$), the condition $n \mid u_{i,j}$ (resp. $n \mid v_{i,j}$) is trivially satisfied, since $n=1$.

We claim that
\begin{equation}\label{eqn:equivH}
H^{(1)}(p^{\mathbf{k}};p^{\mathbf{l}}) = \widetilde{H}^{(1)}(p^{\mathbf{k}};p^{\mathbf{l}})\, q^{k_1+\cdots + k_r}. 
\end{equation}
This equality follows from the definitions of $H^{(1)}(p^{\mathbf{k}};p^{\mathbf{l}})$ and $\widetilde{H}^{(1)}(p^{\mathbf{k}};p^{\mathbf{l}})$, after matching powers of $q$ on each side by applying the following combinatorial lemma.

\begin{lemma}\label{lem:p^ksum}
For each $GT$-pattern $P$, 
 \begin{equation}
 \sum_{i=1}^r k_i(P) = \sum_{i=1}^r \left[ \sum_{j=i}^r v_{i,j} + \sum_{j=i+1}^r u_{i,j} \right].
 \end{equation}
 \end{lemma}

 \begin{proof}
We proceed by expanding each side in terms of the entries $a_{i,j}$ and $b_{i,j}$ in the $GT$-pattern $P$, using the definitions above.  Applying \eqnref{def:ki}, we have
\begin{multline*}
\sum_{i=1}^r k_i(P) = \left[ r\, s_a(0) + \sum_{m=1}^{r-1} s_a(m) + \sum_{i=2}^r \sum_{m=1}^{r+1-i}(2\,s_a(m) + a_{0,m}) - \sum_{i=2}^rs_a(r+1-i) \right] \\
 - \left[ \sum_{m=1}^r s_b(m) + \sum_{i=2}^r \sum_{m=i}^{r+1-i} 2\,s_b(m) \right].
\end{multline*}
Since $\sum_{i=2}^rs_a(r+1-i) = \sum_{m=1}^{r-1}s_a(m)$, the corresponding terms in the first bracket cancel.  After interchanging order of summation and evaluating sums over $i$, we obtain
\begin{equation*}
\sum_{i=1}^r k_i(P) = r\, s_a(0) + \sum_{m=1}^r(r-m)\,a_{0,m} + \sum_{m=1}^{r-1} 2(r-m)\, s_a(m) - \sum_{m=1}^r (1+2(r-m))\, s_b(m).
\end{equation*}
Finally, applying \eqnref{def:si} and combining the first two terms, we conclude that
\begin{equation}\label{eqn:ksum}
\sum_{i=1}^r k_i(P) = \sum_{m=1}^r (2r-m)\, a_{0,m} + \sum_{m=1}^{r-1} \sum_{\ell=m+1}^r 2(r-m)\, a_{m,\ell} - \sum_{m=1}^r \sum_{\ell=m}^r (1+2(r-m))\, b_{m,\ell}.
\end{equation}
On the other hand, from \eqnref{def:uv}, after recombining  terms we have
\begin{multline*}
\sum_{i=1}^r \left[ \sum_{j=i}^r v_{i,j} + \sum_{j=i+1}^r u_{i,j} \right]   = \,\,- \sum_{i=1}^r \left[ b_{i,i} + \sum_{j=i+1}^r \Big(b_{i,j} + 2\sum_{m=i}^r b_{i,m}\Big)\right] \\
+\sum_{i=1}^r \left[ a_{i-1,i} + \sum_{j=i+1}^r \Big( \sum_{m=i}^j \, 2\, a_{i-1,m} + \sum_{m=j+1}^r a_{i-1,m} + \sum_{m=j}^r a_{i,m} \Big) \right]. 
\end{multline*}
After interchanging order of summation and evaluating sums on $j$, this equals 
\begin{align*}
\sum_{i=1}^r &\left[ (1+2(r-i))\, a_{i-1,i} + \sum_{m=i+1}^r (2r+1-(i+m))\, a_{i-1,m} + \sum_{m=i+1}^r (m-i)\, a_{i,m} \right] \\
&-\sum_{i=1}^r \sum_{m=i}^r (1+2(r-i))\, b_{i,m}.
\end{align*}
The $i=1$ terms from the the bracket's first two summands give $\sum_{m=1}^r (2r-m)\, a_{0,m}$, the first term in \eqnref{eqn:ksum}.  After reindexing, the remaining terms in the bracket give $\sum_{i=1}^{r-1} \sum_{m=i+1}^r 2(r-i)\, a_{i,m}$.  Relabeling indices where needed gives the result.  
 \end{proof}

We now manipulate the prime-power supported contributions to the multiple Dirichlet series as in (\ref{pcontrib}). Setting $y_i = |p|^{-2s_i}$ for $i=1,\ldots,r$ and using \eqnref{eqn:equivH}, we have
\begin{equation} \label{firstprep}
	\sum_{\mathbf{k} = (k_1,\ldots,k_r)} \frac{H^{(1)}(p^{k_1}, \ldots, p^{k_r})}{|p|^{2k_1s_1}  \cdots |p|^{2k_rs_r}} 
=  \sum_{\mathbf{k} = (k_1,\ldots,k_r)} \widetilde{H}^{(1)}(p^{k_1}, \ldots, p^{k_r})\, (q y_1)^{k_1} \cdots (q y_r)^{k_r}. \end{equation}
After making the following change of variables:
$$ q y_1 \mapsto x_1^2, \quad q y_2 \mapsto x_1^{-1}x_2, \quad \ldots \quad q y_r \mapsto x_{r-1}^{-1} x_r , $$
the right-hand side of (\ref{firstprep}) becomes
$$ \sum_{(k_1,\ldots,k_r)} \widetilde{H}^{(1)}(p^{k_1}, \ldots, p^{k_r})\, x_1^{2k_1} (x_1^{-1}x_2)^{k_2} \cdots (x_{r-1}^{-1}x_r)^{k_r}. $$
By the relationship between the $k_i$~-coordinates and the weight coordinates $\wgt_i$ given in (\ref{eqn:kweights}), this is just
$$ x_1^{L_1} \cdots x_r^{L_r} \sum_{(k_1,\ldots,k_r)} \widetilde{H}^{(1)}(p^{k_1}, \ldots, p^{k_r})\, x_1^{\wgt_1} \cdots x_r^{\wgt_r}, $$
where the $L_i$ relate to $l_i$ as in (\ref{GTtoprow}). Finally, letting 
$$\gen(P) = \#\{ \text{generic entries in $P$}\}, \text{ and } \max(P) = \#\{ \text{maximal entries in $P$}\} $$ and using the simplifications for $n=1$ in (\ref{gammaareduction}) and (\ref{gammabreduction}) for $\widetilde{H}^{(1)}$ in terms of $\widetilde{G}(P)$, then
\begin{multline} \sum_{\mathbf{k} = (k_1,\ldots,k_r)} \frac{H^{(1)}(p^{k_1}, \ldots, p^{k_r})}{|p|^{2k_1s_1}  \cdots |p|^{2k_rs_r}} \\
 = x_1^{L_1} \cdots x_r^{L_r} \sum_{(k_1,\ldots,k_r)} \left(\frac{-1}{q} \right)^{\max(P)} \left(1-\frac{1}{q} \right)^{\gen(P)} x_1^{\wgt_1} \cdots x_r^{\wgt_r}, \label{eqn:HKprep} \end{multline}
with the $x_i$'s given in terms of $|p|^{-2s_i}$ by the composition of the above changes of variables. The right-hand side of (\ref{eqn:HKprep}) is now amenable to comparison with the identity of Hamel and King.

\subsection{Symplectic Shifted Tableaux}

In order to state the main theorem of Hamel and King (\cite{hamelking}), we must first introduce some additional terminology.  To each strict $GT$-pattern $P$, we may associate an Sp$(2r)$-{\em standard shifted tableau} of shape $\lambda + \rho$.  Below, we follow Hamel and King~{\cite{hamelking}, specializing Definition 2.5 to our circumstances.   We consider the partition $\mu$ of $\lambda + \rho$, whose parts are given by $\mu_i = l_1 + \cdots l_i + r-i+1$, for $i = 1, \ldots, r$.  (These are simply the entries in the top row of the pattern $P$ in $GT(\lambda + \rho)$.)  Such a partition defines a {\em shifted Young diagram} constructed as follows:  $|\mu|$ boxes are arranged in $r$ rows of lengths $\mu_1, \mu_2, \ldots, \mu_r$, and the rows are left-adjusted along a diagonal line.  For instance, if $\mu = (7,4,2,1)$, then our tableau has shape 
$$ \young(\,\,\,\,\,\,\,,:\,\,\,\,,::\,\,,:::\,) $$

It remains to define how the tableau is to be filled.  
The alphabet will consist of the set $A = \{1, 2, \ldots, r\} \cup \{ \overline{1}, \overline{2}, \ldots \, \overline{r} \}$, with ordering $\overline{1} < 1 < \overline{2} < 2 < \cdots < \overline{r} < r$.  We place an entry from $A$ in each of the boxes of the tableau so that the entries are: (1) {\em weakly increasing} from left to right across each row and from top to bottom down each column, and (2) {\em strictly increasing} from top-left to bottom-right along each diagonal.

An explicit correspondence between Sp$(2r)$-standard shifted tableaux and strict $GT$-patterns is given in Definition 5.2 of~\cite{hamelking}.  Below we describe the prescription for determining $S_P$, the tableau corresponding to a given $GT$-pattern $P$, with notation as in \eqnref{eqn:GTpattern}.
\begin{enumerate}
\item For $j=i, \ldots, r$, the entries $a_{i-1,j}$ of $P$ count, respectively, the number of boxes in the $(j-i+1)^{st}$ row of $S_P$ whose entries are less than or equal to the value $r-i+1$.
\item For $j=i, \ldots, r$, the entries $b_{i,j}$ of $P$ count, respectively, the number of boxes in the $(j-i+1)^{st}$ row of $S_P$ whose entries are less than or equal to the value $\overline{r-i+1}$.
\end{enumerate}
An example of this bijection is given in Figure~\ref{gttabbij}.

\begin{figure}[h!]
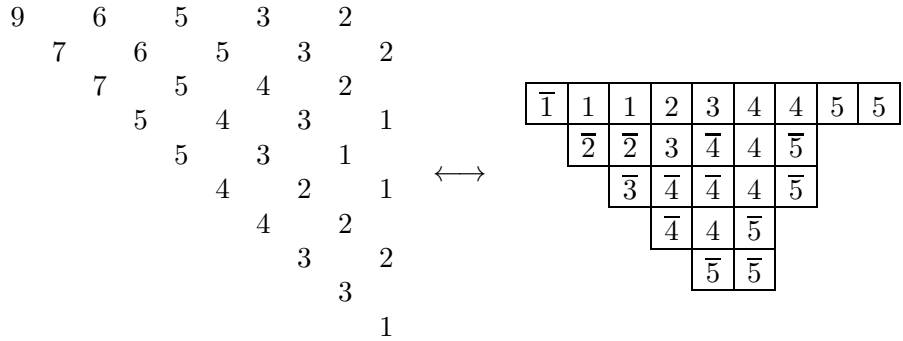

\Yautoscale0
\Yboxdim16pt

\begin{center} 
{\small \begin{tabular}{ccc}
 $ \begin{array}{cccccccccc}
		9 & & 6 & & 5 & & 3 & & 2\\
		& 7 & & 6 & & 5 & & 3 & & 2\\
		& & 7 & & 5 & & 4 & & 2\\
		& & & 5 & & 4 & & 3 & & 1\\
		& & & & 5 & & 3 & & 1\\
		& & & & & 4 & & 2 & & 1\\
		& & & & & & 4 & & 2\\
		& & & & & & & 3 & & 2\\
		& & & & & & & & 3\\
		& & & & & & & & & 1
		\end{array} $ &   
 $ \longleftrightarrow $ & $ \begin{array}{c} \\ \young(\onebar 11234455,:\twobar\twobar 3\fourbar 4\fivebar,::\threebar\fourbar\fourbar 4\fivebar,:::\fourbar 4\fivebar,::::\fivebar\fivebar) \end{array}$
\end{tabular} }
\caption{\small The bijection between $GT$-patterns and symplectic shifted tableaux.} 
\label{gttabbij} 
 \end{center}
 \end{figure}
 Moreover, we associate the following statistics to any symplectic shifted tableau~$S$:
\begin{enumerate}
\item $\wgt(S) = (\wgt_1(S), \wgt_2(S), \ldots, \wgt_r(S))$, where $\wgt_i(S)$ = \# ($i$ entries) - \# ($\overline{i}$ entries).
\item $\con_k(S)$ is the number of connected components of the ribbon strip of $S$ consisting of all the entries $k$.
\item $\row_k(S)$ is the number of rows of $S$ containing an entry $k$, and similarly $\row_{\kbar}(S)$ is the number of rows of $S$ containing an entry $\kbar$.
 \item $\str(S)$ is the total number of connected components of all ribbon strips of $S$.
 \item $\barred(S)$ is the total number of barred entries in $S$.
 \item $\dsp \height(S) = \sum_{k=1}^r (\row_k(S) - \con_k(S) - \row_{\kbar}(S)$).
  \end{enumerate}
It is easy to see that the weights associated with the tableaux $S_P$ are identical to the previously-defined weights associated with the pattern $P$.

The main result of Hamel and King~\cite{hamelking} is the following identity:

\begin{nntheorem}[Hamel-King] Let $\lambda$ be a partition into at most $r$ parts, and let $\rho=(r, r-1, \ldots, 1)$. Then defining
\begin{equation} \label{eqn:HKDdef} D_{Sp(2r)}(\mathbf{x};t) = \prod_{i=1}^r x_i^{r-i+1} \prod_{i=1}^r (1+tx_i^{-2}) \prod_{1 \leq i < j \leq r} (1+tx_i^{-1}x_j)(1+tx_i^{-1}x_j^{-1}),
\end{equation}
and letting $sp_{\lambda}(\mathbf{x}) := sp_{\lambda}(x_1, \ldots, x_r)$ be the character of the highest weight representation of $Sp(2r)$ with highest weight $\lambda$, we have
\begin{equation}\label{eqn:HKidentity2}
D_{Sp(2r)}(t\mathbf{x};t) sp_\lambda(\mathbf{x}) = \sum_{S \in \mathcal{ST}^{\lambda+\rho}(Sp(2r))} t^{\height(S)+r(r+1)/2} (1+t)^{\str(S)-r} \mathbf{x}^{\wgt(S)}.
\end{equation}
where $\mathcal{ST}^{\lambda+\rho}(Sp(2r))$ denotes the set of all Sp$(2r)$-standard shifted tableaux of shape $\lambda + \rho$.
\end{nntheorem} 

\begin{remark} This is a slight variant from the version appearing as Theorem 1.2 of~\cite{hamelking}. They express their identity in the form:
\begin{equation}\label{eqn:HKidentity}
D_{Sp(2r)}(\mathbf{x};t) sp_\lambda(\mathbf{x};t) = \sum_{S \in \mathcal{ST}^{\lambda+\rho}(Sp(2r))} t^{\height(S)+2\barred(S)} (1+t)^{\str(S)-r} \mathbf{x}^{\wgt(S)},
\end{equation}
where $sp_\lambda(\mathbf{x};t)$ is a simple deformation of the usual symplectic character given in (1.13) of~\cite{hamelking}.

In order to relate (\ref{eqn:HKidentity}) to~(\ref{eqn:HKidentity2}), put $x_i \rightarrow tx_i$ for each $i = 1, \ldots, r$, which introduces a factor of $t^{\sum \wgt_i(S)}$ on the right-hand side.  
From the definition of $\wgt(S)$ and the correspondence with $P$, we see that
\begin{equation}
\sum_{i=1}^r \wgt_i(S) = \frac{r(r+1)}{2} - 2 \barred(S) + \sum_{i=1}^r (r-i+1) l_i.
\end{equation}
Moreover, it is a simple exercise to show that 
\begin{equation}
sp_\lambda(t\mathbf{x};t) = t^{\sum (r-i+1)l_i}\, sp_\lambda(\mathbf{x}).
\end{equation}
Applying the previous two identities to (\ref{eqn:HKidentity}) gives the theorem as stated above.
\end{remark}

We now show that the right-hand side of  \eqnref{eqn:HKidentity2} may be expressed in terms of the right-hand side of \eqnref{eqn:HKprep}, leading to an expression for the generating function for $H(p^{k_1}, \ldots, p^{k_r})$ in terms of a symplectic character.
The following lemma relates the exponents in this equation back to our $GT$-pattern $P$ and the statistics of \eqnref{eqn:HKprep}.   

\begin{lemma}\label{lem:HKstats}
Let $P$ be a strict $GT$-pattern of rank $r$ and $S_P$ its associated standard shifted tableau.  Then we have the following relationships:  
\begin{enumerate}
 \item [(a)] $\gen(P) = \str(S_P) - r$,
 \item [(b)] $\max(P) = \height(S_P) + \frac{r(r+1)}{2}$.
\end{enumerate}
\end{lemma}

This is stated without proof implicitly in Corollary 5.3 in \cite{hamelking}, using slightly different notation. The proof is elementary, but we include it in the next section for completeness. Assuming the lemma, letting $t = - \frac{1}{q}$ in \eqnref{eqn:HKidentity2}, and using \eqnref{eqn:HKprep} with $|p|=q$ we see that  
\begin{multline}
\sum_{(k_1,\ldots,k_r)} H(p^{k_1}, \ldots, p^{k_r}) q^{-2k_1s_1} \cdots  q^{-2k_rs_r}\\ = x_1^{L_1} \cdots x_r^{L_r} \, D_{Sp(2r)}(-x_1/q, \ldots, -x_r/q;-1/q) \, sp_\lambda(x_1, \ldots, x_r), \label{almostthere}
\end{multline}
with the identification
\begin{equation} q^{1-2s_1} = x_1^2, \quad q^{1-2s_2} = x_1^{-1}x_2, \quad \ldots \quad q^{1-2s_r} = x_{r-1}^{-1} x_r. \label{satake} \end{equation}
One checks by induction on the rank $r$, that with $x_i$ assigned as above,
$$ x_1 x_2^2 \cdots x_r^r D_{Sp(2r)}(-x_1/q, \ldots, -x_r/q;-1/q) = \prod_{\alpha \in \Phi^+} \left( 1 - q^{-(1+2B(\alpha, \mathbf{s} - \frac{1}{2} \rho^\vee))} \right) $$
with $B(\alpha, \mathbf{s} - \frac{1}{2} \rho^\vee)$ as defined in (\ref{bilinearform}). Moving this product to the left-hand side of (\ref{almostthere}), we may rewrite our above equality as
\begin{multline}  \prod_{\alpha \in \Phi^+} \left( 1 - q^{-(1+2B(\alpha, \mathbf{s} - \frac{1}{2} \rho^\vee))} \right)^{-1} \sum_{(k_1,\ldots,k_r)} H(p^{k_1}, \ldots, p^{k_r}) q^{-2k_1s_1} \cdots  q^{-2k_rs_r} \\ = x_1^{L_1-1} \cdots x_r^{L_r-r} \, sp_\lambda(x_1, \ldots, x_r). \label{justaboutdone} \end{multline}
Note that the terms in the product are precisely the Euler factors for the normalizing zeta factors of $Z_\Psi^{\ast}(\mathbf{s}; \mathbf{m})$ defined in (\ref{zstardefined}) for the case $n=1$. Hence, the terms on the left-hand side of the above equality constitute the complete set of terms in the multiple Dirichlet series $Z_\Psi^{\ast}(\mathbf{s}; \mathbf{m})$ supported at monomials of the form $|p|^{-k_1s_1 - \cdots - k_rs_r}.$ Finally, we can state our second main result.

\begin{bodytheorem} Let $\mathbf{m} = (m_1, \ldots, m_r) \in \mathcal{O}_S$ with $m_i$ non-zero for all $i$. For each prime $p \in \mathcal{O}_S$, let $\ord_p(m_i) = l_i$. Let $H^{(n)}(p^{k_1}, \ldots, p^{k_r}; p^{l_1}, \ldots, p^{l_r})$ with $n=1$ be defined as in Section \ref{specializingsection}. Then the resulting multiple Dirichlet series $Z_\Psi^{\ast}(\mathbf{s}; \mathbf{m})$ agrees with the $(m_1, \ldots, m_r)^{th}$ Fourier-Whittaker coefficient of a minimal parabolic Eisenstein series on $SO_{2r+1}(F_S).$
\end{bodytheorem}

\begin{proof} In the case $n=1$, the multiple Dirichlet series $Z^\ast_\Psi(\mathbf{s};\mathbf{m})$ is Eulerian. Indeed, the power residue symbols used in the definition of twisted multiplicativity in (\ref{cmult}) and (\ref{mmult}) are all trivial. Hence it suffices to check that the Euler factors for $Z^\ast_\Psi$ match those of the corresponding minimal parabolic Eisenstein series at each prime $p \in \mathcal{O}_S$. 

The Euler factors for the minimal parabolic Eisenstein series can be computed using the Casselman-Shalika formula,  Theorem 5.4 in \cite{casselmanshalika}. We briefly recall the form of this expression for a split, reductive group $G$ over a local field $F_v$ with usual Iwasawa decomposition $G = ANK = BK$. Let $\chi$ be an unramified character of the split maximal torus $A$ and consider the induced representation $\text{ind}_B^G(\chi).$ Given an unramified additive character $\psi$ of the unipotent $N^-(F_v)$, opposite the unipotent $N$ of $B$, we have an associated Whittaker functional 
\begin{equation} W_\psi(\phi) = \int_{N^-(F_v)} \phi(\overline{n}) \psi(\overline{n}) d\overline{n} \quad \text{where} \quad \phi(ank) := \chi(a) \delta_B(a)^{1/2} \label{prinserieswhit} \end{equation}
is the normalized spherical vector with $\delta_B$ is the modular quasicharacter.
The associated Whittaker function is given by setting $\mathcal{W}_\phi(g) := W(g\phi)$, and is determined by its value on 
$\pi^{-\lambda}$ for $\lambda \in X_\ast$, the coweight lattice and $\pi$ a uniformizer for $F_v$. Then the Casselman-Shalika formula states that $\mathcal{W}_\phi(\pi^{-\lambda}) = 0$ unless $\lambda$ is dominant, in which case
\begin{equation} \delta_B(\pi^{-\lambda})^{1/2} \mathcal{W}_\phi(\pi^{-\lambda}) = \left ( \prod_{\alpha \in \Phi^+} (1 - q^{-1} \mathbf{t}^{-\alpha^\vee}) \right)  \text{ch}_\lambda(\mathbf{t}) \label{csformula} \end{equation}
where $\text{ch}_\lambda$ is the character of the irreducible representation of the Langlands dual group $G^\vee$ with highest weight $\lambda$ and $\mathbf{t}$ denotes a diagonal representative of the semisimple conjugacy class in $G^\vee$ associated to $\text{ind}_B^G(\chi)$ by Langlands via the Satake isomorphism (see \cite{Borel} for details). In the special case $G = SO(2r+1)$, for relations with the above multiple Dirichlet series, we determine $\mathbf{t} = (x_1, \ldots, x_r)$ according to (\ref{satake}) where $|\pi|_v^{-1} = q$. Since $G^\vee = Sp(2r)$ in this case, the character $\text{ch}_\lambda(\mathbf{t})$ in (\ref{csformula}) is just $sp_\lambda(x_1, \ldots, x_r)$ as in the right-hand side of (\ref{justaboutdone}). Note further that the product over positive roots in (\ref{csformula}) matches the Euler factors for the normalizing zeta factors of $Z^\ast_\Psi$ appearing on the left-hand side of (\ref{justaboutdone}).

While the Casselman-Shalika formula is stated for principal series over a local field, because the global Whittaker coefficient is Eulerian, there is no obstacle to obtaining the analogous global result for $F_S$ from the local result via passage to the adele group. Moreover, the minimal parabolic Eisenstein series Whittaker functional 
$$  \int_{N(\mathbb{A}) / N(F)} E_{\phi}(ng) \psi_{\underline{m}}(n) dn  = \displaystyle \int_{N(\mathbb{A}) / N(F)} \sum_{\gamma \in B(F) \backslash G(F)} \phi(\gamma ng) \psi_{\underline{m}}(n) dn $$
can be shown to match the integral in (\ref{prinserieswhit}) with $\psi = \psi_{\underline{m}}$ by the usual Bruhat decomposition for $G(F)$ and a standard unfolding argument.

Hence according to (\ref{justaboutdone}), the Euler factor for $Z^\ast_\Psi(\mathbf{s}; \mathbf{m})$ matches that of the Fourier-Whittaker coefficient except possibly up to a monomial in the $|p|^{-2s_i}$ with $i=1, \ldots, r$. This disparity arises from the fact that the Whittaker functions in the Casselman-Shalika formula are normalized by the modular quasicharacter $\delta_B^{1/2}$, whereas our multiple Dirichlet series should correspond to unnormalized Whittaker coefficients in accordance with the functional equations $\sigma_i$ as in (\ref{saction}) sending $s_i \mapsto 1-s_i$. Hence, to check that the  right-hand side of (\ref{justaboutdone}) exactly matches the unnormalized Whittaker coefficient of the Eisenstein series, it suffices to verify that $$x_1^{L_1-1} \cdots x_r^{L_r-r} sp_\lambda(x_1, \ldots, x_r)$$ satisfies a local functional equation $\sigma_j$ given in (\ref{saction}) as Dirichlet polynomials in $|p|^{-2s_i}$ for $i=1, \ldots, r$.
\end{proof}

\subsection{Proof of Lemma 4}

\begin{proof}
For part (a) of the lemma, we induct on the rank.  When $r=2$, there are at most six connected components among all the ribbon strips of $S_P$, since $1$ and $\onebar$ may only appear in the top row.  Moreover, since $P$ is strict there must be at least two connected components.  Thus $0 \leq \str(S_P) - 2 \leq 4$.  At each of the four entries in $P$ below the top row. one shows that if the given entry is generic, it increases the count $\str(S_P)$ by $1$.  

Suppose that for a $GT$-pattern of rank $r-1$, each of the $r^2$ entries below the top row increases the count $\str(P)$ by $1$.  Then consider a $GT$-pattern $P$ of rank $r$, and consider the collection of entries $a_{i,j}$, $b_{i,j}$ below the double line.  These entries control the number of connected components consisting of $\onebar$'s, $1$'s, $\ldots \overline{r-1}$'s, and $r-1$'s in~$P$, in precisely the same way as the full collection of entries below the top row in a pattern of rank $r-1$.  Thus inductively, for each generic entry $a_{i,j}$ with $2 \leq i \leq r-1$, $3 \leq j \leq r$ or $b_{i,j}$ with $2 \leq i,j \leq r$, the count $\str(P)$ is increased by~$1$.  Finally, for $i=1$, one easily checks that the value of $\str(S_P)$ is increased by~$1$ for every generic $a_{1,j}$ or $b_{1,j}$.
\bigskip

For (b), we first establish the correct range for $\height(S_P) + \frac{r(r+1)}{2}$.  For each $k$, it is clear that $0 \leq \row_k(S_P) - \con_k(S_P) \leq k-1$ and $0 \leq \row_{\kbar}(S_P) \leq k$.  Combining these inequalities and summing over $k$, we have $0 \leq \height(S_P) + \frac{r(r+1)}{2} \leq r^2$.  We proceed by showing that each of the maximal entries increases the count $\height(S)$ by $1$.  The cases are as follows.

\begin{enumerate}
  \item If $a_{i,j}$ is maximal, then $a_{i,j} = b_{i,j-1}$, hence there are no $\overline{r+1-i}$ entries in row $j-i$ of the tableau.  This decreases $\sum_{k=1}^r \row_{\kbar}(S_P)$ by $1$, hence increasing $\height(S_P)$ by $1$.  
  \item If $b_{i,r}$ is maximal, then $b_{i,r} = 0$, which implies there are no $\overline{r+1-i}$ entries in row $r-i+1$.  This similarly increases $\height(S_P)$ by $1$.  
  \item If $b_{i,j}$ is maximal with $1 \leq j \leq r-1$, then $b_{i,j} = a_{i-1,j+1}$.  Since $P$ is a strict pattern, it must follow that $b_{i,j} < a_{i-1,j}$ and $b_{i,j+1} < a_{i-1,j+1}$.  By these strict inequalities, there are $(r+1-i)$'s in both row $j+1-i$ and row $j+2-i$.  However, by the equality defining $b_{i,j}$ as maximal, the $(r+1-i)$'s in these two rows form one connected component.  (See, for instance, the $\fourbar$ component in the example in Figure 1.)  This decreases $\sum_{k=1}^r \con_k(S_P)$ by $1$, hence increasing $\height(S_P)$ by $1$.        
\end{enumerate}
\end{proof}

\end{document}